\documentclass[11pt]{article}
\pdfoutput=1

\usepackage{graphicx}
\usepackage{amsmath}
\usepackage{amssymb,amsfonts}
\usepackage{amsthm}
\usepackage{cite}
\usepackage{bm}
\newcommand{\norm}[1]{\left\lVert#1\right\rVert}

\usepackage{mathrsfs}

\usepackage{amssymb}   
\usepackage{multirow}

\usepackage[margin=1in]{geometry}

\newtheorem{lemma}{Lemma}
\newtheorem{theorem}{Theorem}
\newtheorem{proposition}{Proposition}

\theoremstyle{definition}

\newtheorem{example}{Example}

\usepackage{authblk}

\usepackage{hyperref}
\usepackage[nameinlink]{cleveref}

\title{A Superconvergent HDG Method for Distributed Control of Convection Diffusion PDEs}

\author[1]{Weiwei Hu\thanks{weiwei.hu@okstate.edu}}
\author[2]{Jiguang Shen\thanks{shenx179@umn.edu}}
\author[3]{John R.~Singler\thanks{singlerj@mst.edu}}
\author[3]{Yangwen Zhang\thanks{ywzfg4@mst.edu}}
\author[4]{Xiaobo Zheng\thanks{zhengxiaobosc@yahoo.com}}
\affil[1]{Department of Mathematics, Oklahoma State University, Stillwater, OK}
\affil[2]{School  of Mathematics, University of Minnesota, MN}
\affil[3]{Department of Mathematics and Statistics, Missouri University of Science and Technology, Rolla, MO}
\affil[4]{College of Mathematics, Sichuan University, Chengdu, China}

\begin{document}

\maketitle

\begin{abstract}
	We consider a distributed optimal control problem governed by an elliptic convection diffusion PDE, and propose a hybridizable discontinuous Galerkin (HDG) method to approximate the solution.  We use polynomials of degree $k+1$ and $k \ge 0$ to approximate the state, dual state, and their fluxes, respectively. Moreover, we use polynomials of degree $k$ to approximate the numerical traces of the state and dual state on the faces,  which are the only globally coupled unknowns.  We prove optimal a priori error estimates for all variables when $ k > 0 $.  Furthermore, from the point of view of the number of degrees of freedom of the globally coupled unknowns, this method achieves superconvergence for the state, dual state, and control when $k\geq 1$. We illustrate our convergence results with numerical experiments. 
\end{abstract}

\section{Introduction}
\label{intro}
We consider the following  distributed control problem. Let $\Omega\subset \mathbb{R}^{d} $ $ (d\geq 2)$ be a Lipschitz polyhedral domain  with boundary $\Gamma = \partial \Omega$.  The goal is to minimize
\begin{align}
J(u)=\frac{1}{2}\| y- y_{d}\|^2_{L^{2}(\Omega)}+\frac{\gamma}{2}\|u\|^2_{L^{2}(\Omega)}, \quad \gamma>0, \label{cost1}
\end{align}
subject to
\begin{equation}\label{Ori_problem}
\begin{split}
-\Delta y+\bm \beta\cdot\nabla y&=f+u \quad\text{in}~\Omega,\\
y&=g\qquad\quad\text{on}~\partial\Omega,
\end{split}
\end{equation}
where $ f \in L^2(\Omega) $ and the vector field $\bm{\beta}$ satisfies
\begin{align}\label{beta_con}
\nabla\cdot\bm{\beta} \le 0.
\end{align}
It is well known that this optimal control problem is equivalent to the optimality system
\begin{subequations}\label{eq_adeq}
	\begin{align}
	-\Delta y+\bm \beta\cdot\nabla y &=f+u\quad~\text{in}~\Omega,\label{eq_adeq_a}\\
	y&=g\qquad~~~~\text{on}~\partial\Omega,\label{eq_adeq_b}\\
	-\Delta z-\nabla\cdot(\bm{\beta} z) &=y_d-y\quad~\text{in}~\Omega,\label{eq_adeq_c}\\
	z&=0\qquad\quad~~\text{on}~\partial\Omega,\label{eq_adeq_d}\\
	z-\gamma  u&=0\qquad\quad~~\text{in}~\Omega.\label{eq_adeq_e}
	\end{align}
\end{subequations}

Many different numerical methods have been investigated for this type of problem including approaches based on the finite element method \cite{MR2851444,MR2719819,MR2475653,MR2302057,MR3144702,MR2486088,MR2595051,MR3451479,MR2068903}, mixed finite elements \cite{MR2550371,MR2851444,MR2971662}, and discontinuous Galerkin (DG) methods \cite{MR2595051,MR2587414,MR2773301,MR3416418,MR3149415,MR3022208,MR2644299}.  Also, hybridizable discontinuous Galerkin (HDG) methods have recently been explored for various optimal control problems for the Poisson equation \cite{MR3508834,HuShenSinglerZhangZheng_HDG_Dirichlet_control1} and the above convection diffusion equation \cite{HuShenSinglerZhangZheng_HDG_Dirichlet_control5}.

In this earlier work \cite{HuShenSinglerZhangZheng_HDG_Dirichlet_control5}, we used a hybridizable discontinuous Galerkin (HDG) method to approximate the solution of the optimality system \eqref{eq_adeq}.  We used polynomials of degree $k$ to approximate all variables and obtained optimal convergence rates when $ \bm \beta $ is divergence free.

In this work, we investigate a different HDG method for the above problem and prove that it is superconvergent.  Specifically, we use polynomials of degree $k+1$ to approximate the state $y$ and dual state $z$ and polynomials of degree $k \ge 0$ for the fluxes $\bm q = -\nabla  y $ and $ \bm p = -\nabla z$. Moreover, we only use polynomials of degree $k$ to approximate the numerical traces of the state and dual state on the faces,  which are the only globally coupled unknowns.  We describe the method in Section \ref{sec:background}, and then in Section \ref{sec:analysis} we obtain the a priori error bounds
\begin{align*}
&\norm{y-{y}_h}_{0,\Omega}=O( h^{k+1+\min\{k,1\}} ),\quad  \;\norm{z-{z}_h}_{0,\Omega}=O( h^{k+1+\min\{k,1\}} ),\\
&\norm{\bm{q}-\bm{q}_h}_{0,\Omega} = O( h^{k+1} ),\qquad  \qquad\;\; \norm{\bm{p}-\bm{p}_h}_{0,\Omega} = O( h^{k+1} ),
\end{align*}
and
\begin{align*}
&\norm{u-{u}_h}_{0,\Omega} = O( h^{k+1+\min\{k,1\}} ).
\end{align*}
From the point of view of the global degrees of freedom, we obtain superconvergent approximations to $y$, $z$, and $u$ without postprocessing if $k\ge 1$. We demonstrate the performance of the HDG method with numerical experiments in Section \ref{sec:numerics}.

\section{HDG scheme for the optimal control problem}
\label{sec:background}

We begin with notation and a complete description of the HDG method.

	\subsection{Notation}
	
Throughout this work we adopt the standard notation $W^{m,p}(\Omega)$ for Sobolev spaces on $\Omega$ with norm $\|\cdot\|_{m,p,\Omega}$ and seminorm $|\cdot|_{m,p,\Omega}$. We denote $W^{m,2}(\Omega)$ by $H^{m}(\Omega)$ with norm $\|\cdot\|_{m,\Omega}$ and seminorm $|\cdot|_{m,\Omega}$. We also set $H_0^1(\Omega)=\{v\in H^1(\Omega):v=0 \;\mbox{on}\; \partial \Omega\}$ and $ H(\text{div},\Omega) = \{\bm{v}\in [L^2(\Omega)]^d, \nabla\cdot \bm{v}\in L^2(\Omega)\} $.  We denote the $L^2$-inner products on $L^2(\Omega)$ and $L^2(\Gamma)$ by
\begin{align*}
(v,w) &= \int_{\Omega} vw  \quad \forall v,w\in L^2(\Omega),\\
\left\langle v,w\right\rangle &= \int_{\Gamma} vw  \quad\forall v,w\in L^2(\Gamma).
\end{align*}

Let $\mathcal{T}_h$ be a collection of disjoint elements that partition $\Omega$, and let $\partial \mathcal{T}_h$ be the set $\{\partial K: K\in \mathcal{T}_h\}$. For an element $K  \in \mathcal{T}_h$, let $e = \partial K \cap \Gamma$ denote the boundary face of $ K $ if the $d-1$ Lebesgue measure of $e$ is non-zero. For two elements $K^+$ and $K^-$ in $\mathcal{T}_h$, let $e = \partial K^+ \cap \partial K^-$ denote the interior face between $K^+$ and $K^-$ if the $d-1$ Lebesgue measure of $e$ is non-zero. Let $\varepsilon_h^o$ and $\varepsilon_h^{\partial}$ denote the set of interior and boundary faces, respectively, and let $\varepsilon_h$ be the union of  $\varepsilon_h^o$ and $\varepsilon_h^{\partial}$. Furthermore, we introduce
\begin{align*}
(w,v)_{\mathcal{T}_h} = \sum_{K\in\mathcal{T}_h} (w,v)_K,   \quad\quad\quad\quad\left\langle \zeta,\rho\right\rangle_{\partial\mathcal{T}_h} = \sum_{K\in\mathcal{T}_h} \left\langle \zeta,\rho\right\rangle_{\partial K}.
\end{align*}

Let $\mathcal{P}^k(D)$ denote the set of polynomials of degree at most $k$ on a domain $D$.  We use the discontinuous finite element spaces
\begin{align}
\bm{V}_h  &:= \{\bm{v}\in [L^2(\Omega)]^d: \bm{v}|_{K}\in [\mathcal{P}^k(K)]^d, \forall K\in \mathcal{T}_h\},\\
{W}_h  &:= \{{w}\in L^2(\Omega): {w}|_{K}\in \mathcal{P}^{k+1}(K), \forall K\in \mathcal{T}_h\},\\
{M}_h  &:= \{{\mu}\in L^2(\mathcal{\varepsilon}_h): {\mu}|_{e}\in \mathcal{P}^k(e), \forall e\in \varepsilon_h\}.
\end{align}
%
Let  $M_h(o)$ and $M_h(\partial)$ denote the spaces of discontinuous finite element functions of polynomial degree at most $ k $ defined on the set of interior faces $\varepsilon_h^o$ and boundary faces $\varepsilon_h^{\partial}$, respectively.  For any functions $w\in W_h$ and $\bm r\in \bm V_h$, let $\nabla w$ and $ \nabla \cdot \bm r $ denote the piecewise gradient and divergence on each element $K\in \mathcal T_h$.

	\subsection{The HDG Formulation}

For the HDG method, we consider a mixed formulation of the optimality system \eqref{eq_adeq} and approximate the state $ y $, the dual state $ z $, the fluxes $ \bm q = -\nabla y $ and $ \bm p = -\nabla z $, and the numerical traces of $ y $ and $ z $ on the faces.  The approximate optimal distributed control is found directly using a discrete version of the optimality condition \eqref{eq_adeq_e}.  One important feature of HDG methods is the local solver:  The unknowns corresponding to all variables except the numerical traces can be eliminated locally on each element, which leads to a globally coupled system involving only the coefficients of the numerical traces.  This leads to a reduction in the computational cost.  For more information on HDG methods, see, e.g., \cite{MR2485455,MR2772094,MR2513831,MR2558780,MR2796169,MR3626531,MR3522968,MR3463051,MR3452794,MR3343926}.


The mixed weak form of the optimality system \eqref{eq_adeq_a}-\eqref{eq_adeq_e} is given by
\begin{subequations}\label{mixed}
	\begin{align}
	(\bm q,\bm r_1)-( y,\nabla\cdot \bm r_1)+\langle y,\bm r_1\cdot \bm n\rangle&=0,\label{mixed_a}\\
	(\nabla\cdot(\bm q+\bm \beta y),  w_1)-(\nabla\cdot\bm\beta y,w_1)&= ( f+ u, w_1),  \label{mixed_b}\\
	(\bm p,\bm r_2)-(z,\nabla \cdot\bm r_2)+\langle z,\bm r_2\cdot\bm n\rangle&=0,\label{mixed_c}\\
	(\nabla\cdot(\bm p-\bm \beta z), \bm w_2)&= (y_d- y, w_2),  \label{mixed_d}\\
	( z-\gamma u,v)&=0,\label{mixed_e}
	\end{align}
\end{subequations}
for all $(\bm r_1, w_1,\bm r_2, w_2,v)\in H(\text{div},\Omega)\times L^2(\Omega)\times H(\text{div},\Omega)\times L^2(\Omega)\times L^2(\Omega)$.  To approximate the solution of this problem, the HDG method seeks approximate fluxes ${\bm{q}}_h,{\bm{p}}_h \in \bm{V}_h $, states $ y_h, z_h \in W_h $, interior element boundary traces $ \widehat{y}_h^o,\widehat{z}_h^o \in M_h(o) $, and  control $ u_h \in W_h$ satisfying
\begin{subequations}\label{HDG_discrete2}
	\begin{align}
	(\bm q_h,\bm r_1)_{\mathcal T_h}-( y_h,\nabla\cdot\bm r_1)_{\mathcal T_h}+\langle \widehat y_h^o,\bm r_1\cdot\bm n\rangle_{\partial\mathcal T_h\backslash \varepsilon_h^\partial}&=-\langle  g,\bm r_1\cdot\bm n\rangle_{\varepsilon_h^\partial}, \label{HDG_discrete2_a}\\
	-(\bm q_h+\bm \beta y_h,  \nabla w_1)_{\mathcal T_h}-(\nabla\cdot\bm\beta y_h,w_1)_{\mathcal T_h} -  ( u_h, w_1)_{\mathcal T_h}  \quad  \nonumber \\ 
	+\langle\widehat {\bm q}_h\cdot\bm n,w_1\rangle_{\partial\mathcal T_h} +\langle \bm \beta\cdot\bm n\widehat y_h^o,w_1\rangle_{\partial\mathcal T_h\backslash\varepsilon_h^\partial}  &=  ( f, w_1)_{\mathcal T_h} \ \nonumber\\
	&\quad- \langle \bm \beta\cdot\bm n g,w_1\rangle_{\varepsilon_h^\partial}, \label{HDG_discrete2_b}
	\end{align}
	for all $(\bm{r}_1, w_1)\in \bm{V}_h\times W_h$,
	\begin{align}
	(\bm p_h,\bm r_2)_{\mathcal T_h}-(z_h,\nabla\cdot\bm r_2)_{\mathcal T_h}+\langle \widehat z_h,\bm r_2\cdot\bm n\rangle_{\partial\mathcal T_h\backslash\varepsilon_h^\partial}&=0,\label{HDG_discrete2_c}\\
	-(\bm p_h-\bm \beta z_h, \nabla w_2)_{\mathcal T_h}+\langle\widehat{\bm p}_h\cdot\bm n,w_2\rangle_{\partial\mathcal T_h}  \quad  \nonumber\\
	 -\langle\bm \beta\cdot\bm n\widehat z_h^o,w_2\rangle_{\partial\mathcal T_h\backslash \varepsilon_h^\partial} +  ( y_h, w_2)_{\mathcal T_h}&= (y_d, w_2)_{\mathcal T_h},  \label{HDG_discrete2_d}
	\end{align}
	for all $(\bm{r}_2, w_2)\in \bm{V}_h\times W_h$,
	\begin{align}
	\langle\widehat {\bm q}_h\cdot\bm n+\bm \beta\cdot\bm n\widehat y_h^o,\mu_1\rangle_{\partial\mathcal T_h\backslash\varepsilon^{\partial}_h}&=0\label{HDG_discrete2_e},
	\end{align}
	for all $\mu_1\in M_h(o)$,
	\begin{align}
	\langle\widehat{\bm p}_h\cdot\bm n-\bm \beta\cdot\bm n\widehat z_h^o,\mu_2\rangle_{\partial\mathcal T_h\backslash\varepsilon^{\partial}_h}&=0,\label{HDG_discrete2_f}
	\end{align}
	for all $\mu_2\in M_h(o)$, and the optimality condition 
	\begin{align}
	(z_h-\gamma u_h, w_3)_{\mathcal T_h} &= 0\label{HDG_discrete2_g},
	\end{align}
	for all $ w_3\in W_h$.  The numerical traces on $\partial\mathcal{T}_h$ are defined by 
	\begin{align}
	\widehat{\bm{q}}_h\cdot \bm n &=\bm q_h\cdot\bm n+h^{-1} (P_My_h-\widehat y_h^o) + \tau_1(y_h-\widehat y_h^o)   \qquad ~~~ \mbox{on} \; \partial \mathcal{T}_h\backslash\varepsilon_h^\partial, \label{HDG_discrete2_h}\\
	\widehat{\bm{q}}_h\cdot \bm n &=\bm q_h\cdot\bm n+h^{-1} (P_My_h-P_Mg) +  \tau_1(y_h- P_M g)  \quad \mbox{on}\;  \varepsilon_h^\partial, \label{HDG_discrete2_i}\\
	\widehat{\bm{p}}_h\cdot \bm n &=\bm p_h\cdot\bm n+h^{-1}(P_Mz_h-\widehat z_h^o) + \tau_2(z_h-\widehat y_h^o) \qquad~~~ \mbox{on} \; \partial \mathcal{T}_h\backslash\varepsilon_h^\partial,\label{HDG_discrete2_j}\\
	\widehat{\bm{p}}_h\cdot \bm n &=\bm p_h\cdot\bm n+h^{-1} P_Mz_h+\tau_2 z_h\qquad\qquad\qquad\qquad\quad\mbox{on}\;  \varepsilon_h^\partial,\label{HDG_discrete2_k}
	\end{align}
\end{subequations}
where $\tau_1$ and $\tau_2$ are stabilization functions defined on $\partial \mathcal T_h$.  In the next section, we give conditions that the stabilization functions must satisfy in order to guarantee the convergence results.

The implementation of the above HDG method and the local solver is similar to the implementation of another HDG method described in our recent work \cite{HuShenSinglerZhangZheng_HDG_Dirichlet_control5}; therefore, we omit the details.

\section{Error Analysis}
\label{sec:analysis}

Next, we perform an error analysis of the above HDG method.  Throughout this section, we assume $ \Omega $ is a bounded convex polyhedral domain, $ \bm \beta $ is continuous on $ \bar{\Omega} $, $ \bm \beta \in [ W^{1,\infty}(\Omega) ]^d $, and the solution of the optimality system \eqref{eq_adeq} is sufficiently smooth.

We choose the stabilization functions $\tau_1$ and $\tau_2$ so that the following conditions are satisfied:
\begin{description}
	
	\item[\textbf{(A1)}] $\tau_1 = \tau_2 + \bm{\beta}\cdot \bm n$.
	
	\item[\textbf{(A2)}] For any  $K\in\mathcal T_h$, $\min{(\tau_1-\frac 1 2 \bm \beta \cdot \bm n)}|_{\partial K} >0$.
	
\end{description}
Note that \textbf{(A1)} and \textbf{(A2)} imply
\begin{equation}\label{eqn:tau1_condition}
\min{(\tau_2 + \frac 1 2 \bm \beta \cdot \bm n)}|_{\partial K} >0  \quad  \mbox{for any $K\in\mathcal T_h$.}
\end{equation}

Below, we prove the main result:
\begin{theorem}\label{main_res}
	We have
		\begin{align*}
		\|\bm q-\bm q_h\|_{\mathcal T_h}&\lesssim h^{k+1}(|\bm q|_{k+1}+|y|_{k+2}+|\bm p|_{k+1}+|z|_{k+2}),\\
		\|\bm p-\bm p_h\|_{\mathcal T_h}&\lesssim h^{k+1}(|\bm q|_{k+1}+|y|_{k+2}+|\bm p|_{k+1}+|z|_{k+2}),\\
		\|y-y_h\|_{\mathcal T_h}&\lesssim h^{k+1+\min\{k,1\}}(|\bm q|_{k+1}+|y|_{k+2}+|\bm p|_{k+1}+|z|_{k+2}),\\
		\|z-z_h\|_{\mathcal T_h}&\lesssim h^{k+1+\min\{k,1\}}(|\bm q|_{k+1}+|y|_{k+2}+|\bm p|_{k+1}+|z|_{k+2}),\\
		\|u-u_h\|_{\mathcal T_h}&\lesssim h^{k+1+\min\{k,1\}}(|\bm q|_{k+1}+|y|_{k+2}+|\bm p|_{k+1}+|z|_{k+2}).
\end{align*}
\end{theorem}

	\subsection{Preliminary material}


Let $\bm\Pi :  [L^2(\Omega)]^d \to \bm V_h$, $\Pi :  L^2(\Omega) \to W_h$, and $P_M:  L^2(\varepsilon_h) \to M_h$ denote the standard $L^2$ projections, which satisfy
\begin{equation}\label{L2_projection}
\begin{split}
(\bm\Pi \bm q,\bm r)_{K} &= (\bm q,\bm r)_{K} ,\qquad \forall \bm r\in [{\mathcal P}_{k}(K)]^d,\\
(\Pi y,w)_{K}  &= (y,w)_{K} ,\qquad \forall w\in \mathcal P_{k+1}(K),\\
\left\langle P_M m, \mu\right\rangle_{ e} &= \left\langle  m, \mu\right\rangle_{e }, \quad\;\;\; \forall \mu\in \mathcal P_{k}(e).
\end{split}
\end{equation}
We use the following well-known bounds:
%
\begin{subequations}\label{classical_ine}
	\begin{align}
	\norm {\bm q -\bm{\Pi q}}_{\mathcal T_h} &\lesssim h^{k+1} \norm{\bm q}_{k+1,\Omega},\quad \norm {y -{\Pi y}}_{\mathcal T_h} \lesssim h^{k+2} \norm{y}_{k+2,\Omega},\\
	\norm {y -{\Pi y}}_{\partial\mathcal T_h} &\lesssim  h^{k+\frac 3 2} \norm{y}_{k+2,\Omega},
	\quad \norm {\bm q\cdot \bm n -\bm{\Pi q}\cdot \bm n}_{\partial \mathcal T_h} \lesssim h^{k+\frac 12} \norm{\bm q}_{k+1,\Omega},\\
	\norm {w}_{\partial \mathcal T_h} &\lesssim h^{-\frac 12} \norm {w}_{ \mathcal T_h}, \quad\qquad \forall w\in W_h.
	\end{align}
\end{subequations}
We have the same projection error bounds for $\bm p$ and $z$.

Next, define HDG operators $ \mathscr B_1$ and $ \mathscr B_2 $ by
\begin{align}
\hspace{1em}&\hspace{-1em}\mathscr  B_1( \bm q_h,y_h,\widehat y_h^o;\bm r_1,w_1,\mu_1) \nonumber\\
& = (\bm q_h,\bm r_1)_{\mathcal T_h}-( y_h,\nabla\cdot\bm r_1)_{\mathcal T_h}+\langle \widehat y_h^o,\bm r_1\cdot\bm n\rangle_{\partial\mathcal T_h\backslash \varepsilon_h^\partial}-(\bm q_h+\bm \beta y_h,  \nabla w_1)_{\mathcal T_h} \nonumber\\ 
& \quad -(\nabla\cdot\bm\beta y_h,w_1)_{\mathcal T_h} +\langle {\bm q}_h\cdot\bm n +h^{-1}P_My_h +\tau_1 y_h,w_1\rangle_{\partial\mathcal T_h}\nonumber\\ 
&\quad+\langle (\bm\beta\cdot\bm n -h^{-1}-\tau_1) \widehat y_h^o,w_1\rangle_{\partial\mathcal T_h\backslash \varepsilon_h^\partial}\nonumber\\
& \quad -\langle  {\bm q}_h\cdot\bm n+\bm \beta\cdot\bm n\widehat y_h^o +h^{-1}(P_My_h-\widehat y_h^o) + \tau_1(y_h - \widehat y_h^o),\mu_1\rangle_{\partial\mathcal T_h\backslash\varepsilon^{\partial}_h},\label{def_B1}\\
\hspace{1em}&\hspace{-1em}\mathscr B_2(\bm p_h,z_h,\widehat z_h^o;\bm r_2, w_2,\mu_2) \nonumber\\
&=(\bm p_h,\bm r_2)_{\mathcal T_h}-( z_h,\nabla\cdot\bm r_2)_{\mathcal T_h}+\langle \widehat z_h^o,\bm r_2\cdot\bm n\rangle_{\partial\mathcal T_h\backslash\varepsilon_h^\partial}-(\bm p_h-\bm \beta z_h,  \nabla w_2)_{\mathcal T_h}\nonumber\\
& \quad +\langle {\bm p}_h\cdot\bm n +h^{-1} P_Mz_h +\tau_2 z_h ,w_2\rangle_{\partial\mathcal T_h} -\langle (\bm \beta\cdot\bm n + h^{-1}+\tau_2)\widehat z_h^o ,w_2\rangle_{\partial\mathcal T_h\backslash\varepsilon_h^\partial}\nonumber\\
& \quad -\langle  {\bm p}_h\cdot\bm n-\bm \beta\cdot\bm n\widehat z_h^o +h^{-1}(P_Mz_h-\widehat z_h^o)+\tau_2(z_h - \widehat z_h^o),\mu_2\rangle_{\partial\mathcal T_h\backslash\varepsilon^{\partial}_h}\label{def_B2}.
\end{align}
We use $\mathscr B_1$ and $\mathscr B_2$ to rewrite the HDG discretization of the optimality system \eqref{HDG_discrete2}: find $$({\bm{q}}_h,{\bm{p}}_h,y_h,z_h,u_h,\widehat y_h^o,\widehat z_h^o)\in \bm{V}_h\times\bm{V}_h\times W_h \times W_h\times W_h\times M_h(o)\times M_h(o)$$ satisfying
\begin{subequations}\label{HDG_full_discrete}
	\begin{align}
	\mathscr B_1(\bm q_h,y_h,\widehat y_h^o;\bm r_1,w_1,\mu_1)&=( f+ u_h, w_1)_{\mathcal T_h} - \langle P_M g,\bm r_1\cdot\bm n\rangle \nonumber \\
	& \quad -\langle (\bm\beta\cdot\bm n-h^{-1}-\tau_1) P_M g,w_1\rangle_{\varepsilon_h^\partial},\label{HDG_full_discrete_a}\\
	\mathscr B_2(\bm p_h,z_h,\widehat z_h^o;\bm r_2,w_2,\mu_2)&=(y_d-y_h,w_2)_{\mathcal T_h},\label{HDG_full_discrete_b}\\
	(z_h-\gamma u_h,w_3)_{\mathcal T_h}&= 0,\label{HDG_full_discrete_e}
	\end{align}
\end{subequations}
for all $\left(\bm{r}_1, \bm{r}_2,w_1,w_2,w_3,\mu_1,\mu_2\right)\in \bm{V}_h\times\bm{V}_h\times W_h \times W_h\times W_h\times M_h(o)\times M_h(o)$.

Next, we prove an energy identity for the HDG operators and prove the discrete optimality system \eqref{HDG_full_discrete} is well-posed.  The proofs of the next three results are similar to the proofs of the corresponding results in our earlier work \cite{HuShenSinglerZhangZheng_HDG_Dirichlet_control5}; we include them for completeness.
\begin{lemma}\label{property_B}
	For any $ ( \bm v_h, w_h, \mu_h ) \in \bm V_h \times W_h \times M_h(o) $, we have
	\begin{align*}
	\hspace{2em}&\hspace{-2em} \mathscr B_1(\bm v_h,w_h,\mu_h;\bm v_h,w_h,\mu_h)\\
	&=(\bm v_h,\bm v_h)_{\mathcal T_h}+ \langle (\tau_1 - \frac 12 \bm \beta\cdot\bm n)(w_h-\mu_h),w_h-\mu_h\rangle_{\partial\mathcal T_h\backslash \varepsilon_h^\partial}-\frac 12 (\nabla\cdot\bm\beta w_h,w_h)_{\mathcal T_h}\\
	&\quad+\langle h^{-1}(P_Mw_h-\mu_h),P_Mw_h-\mu_h\rangle_{\partial\mathcal T_h\backslash \varepsilon_h^\partial}+\langle (\tau_1-\frac12\bm \beta\cdot\bm n) w_h,w_h\rangle_{\varepsilon_h^\partial}\\
	&\quad+\langle h^{-1} P_Mw_h,P_Mw_h\rangle_{\varepsilon_h^\partial},\\
	\hspace{2em}&\hspace{-2em}\mathscr B_2(\bm v_h,w_h,\mu_h;\bm v_h,w_h,\mu_h)\\
	&=(\bm v_h,\bm v_h)_{\mathcal T_h}+ \langle (\tau_2 + \frac 12 \bm \beta\cdot\bm n)(w_h-\mu_h),w_h-\mu_h\rangle_{\partial\mathcal T_h\backslash \varepsilon_h^\partial}-\frac 12 (\nabla\cdot\bm\beta w_h,w_h)_{\mathcal T_h}\\
	&\quad+\langle h^{-1}(P_Mw_h-\mu_h),P_Mw_h-\mu_h \rangle_{\partial\mathcal T_h\backslash \varepsilon_h^\partial}+\langle (\tau_2+\frac12\bm \beta\cdot\bm n) w_h,w_h\rangle_{\varepsilon_h^\partial}\\
	&\quad+\langle h^{-1} P_Mw_h,P_Mw_h\rangle_{\varepsilon_h^\partial}.
	\end{align*}
\end{lemma}
\begin{proof}
	We prove the first identity; the proof of the second identity is similar.
	\begin{align*}
	\hspace{1em}&\hspace{-1em} 	\mathscr B_1(\bm v_h,w_h,\mu_h;\bm v_h,w_h,\mu_h)\\
	&=(\bm v_h,\bm v_h)_{\mathcal T_h}-( w_h,\nabla\cdot\bm v_h)_{\mathcal T_h}+\langle \mu_h,\bm v_h\cdot\bm n\rangle_{\partial\mathcal T_h\backslash \varepsilon_h^\partial}-(\bm v_h+\bm \beta w_h,  \nabla w_h)_{\mathcal T_h}\\
	& \quad -(\nabla\cdot\bm\beta w_h,w_h)_{\mathcal T_h}+\langle {\bm v}_h\cdot\bm n +h^{-1}P_Mw_h+\tau_1w_h,w_h\rangle_{\partial\mathcal T_h} \\
	& \quad+\langle (\bm\beta\cdot\bm n -h^{-1}-\tau_1) \mu_h,w_h\rangle_{\partial\mathcal T_h\backslash \varepsilon_h^\partial} \\
	&\quad-\langle  {\bm v}_h\cdot\bm n+\bm \beta\cdot\bm n\mu_h +h^{-1}(P_Mw_h-\mu_h) + \tau_1(w_h - \mu_h),\mu_h \rangle_{\partial\mathcal T_h\backslash\varepsilon^{\partial}_h},\\
	&=(\bm v_h,\bm v_h)_{\mathcal T_h}-(\bm \beta w_h,  \nabla w_h)_{\mathcal T_h} -(\nabla\cdot\bm\beta w_h,w_h)_{\mathcal T_h}\\
	&\quad+\langle  h^{-1} P_Mw_h +\tau_1 w_h,w_h\rangle_{\partial\mathcal T_h} +\langle (\bm\beta\cdot\bm n -h^{-1}-\tau_1) \mu_h,w_h\rangle_{\partial\mathcal T_h\backslash \varepsilon_h^\partial}\\
	&\quad-\langle  \bm \beta\cdot\bm n \mu_h +h^{-1}(P_Mw_h-\mu_h) + \tau_1(w_h - \mu_h ),\mu_h\rangle_{\partial\mathcal T_h\backslash\varepsilon^{\partial}_h}.
	\end{align*}
	For the second term, we have
	\begin{align*}
	(\bm \beta w_h,\nabla w_h)_{\mathcal T_h}&=(\bm \beta\cdot\nabla w_h,w_h)_{\mathcal T_h}=(\nabla\cdot(\bm \beta w_h),w_h)_{\mathcal T_h}-(\nabla\cdot\bm \beta w_h,w_h)_{\mathcal T_h}\\
	&=\langle\bm \beta\cdot\bm n w_h,w_h\rangle_{\partial\mathcal T_h}-(\bm \beta w_h,\nabla w_h)_{\mathcal T_h}-(\nabla\cdot\bm \beta w_h,w_h)_{\mathcal T_h},
	\end{align*}
	which implies
	\begin{align}\label{beta_iden}
	(\bm \beta w_h,\nabla w_h)_{\mathcal T_h}=\frac12\langle\bm \beta\cdot\bm n w_h,w_h\rangle_{\partial\mathcal T_h}-\frac12(\nabla\cdot\bm \beta w_h,w_h)_{\mathcal T_h}.
	\end{align}
	This gives
	\begin{align*}
	\hspace{1em}&\hspace{-1em}  \mathscr B_1 (\bm v_h,w_h,\mu_h;\bm v_h,w_h,\mu_h)\\
	&=(\bm v_h,\bm v_h)_{\mathcal T_h}+ \langle (\tau_1 - \frac 12 \bm \beta\cdot\bm n)(w_h-\mu_h),w_h-\mu_h\rangle_{\partial\mathcal T_h\backslash \varepsilon_h^\partial} - \frac 1 2(\nabla\cdot\bm\beta w_h,w_h)_{\mathcal T_h}\\
	&\quad+\langle h^{-1}(P_Mw_h-\mu_h),P_Mw_h-\mu_h \rangle_{\partial\mathcal T_h\backslash \varepsilon_h^\partial}+\langle (\tau_1-\frac12\bm \beta\cdot\bm n) w_h,w_h\rangle_{\varepsilon_h^\partial}\\
	& \quad+\langle h^{-1} P_Mw_h,P_Mw_h\rangle_{\varepsilon_h^\partial} -\frac12\langle\bm \beta\cdot\bm n \mu_h,\mu_h\rangle_{\partial\mathcal T_h\backslash \varepsilon_h^\partial}.
	\end{align*}
	Since $ \mu_h$ is single-valued across the interfaces, we have
	\begin{align*}
	-\frac12\langle\bm \beta\cdot\bm n\mu_h,\mu_h\rangle_{\partial\mathcal T_h\backslash\varepsilon_h^\partial}=0.
	\end{align*}
	This completes the proof.
\end{proof}

The following property of the HDG operators is crucial to our analysis.
\begin{lemma}\label{identical_equa}
	We have $\mathscr B_1 (\bm q_h,y_h,\widehat y_h^o;\bm p_h,-z_h,-\widehat z_h^o) + \mathscr B_2 (\bm p_h,z_h,\widehat z_h^o;-\bm q_h,y_h,\widehat y_h^o) = 0.$
\end{lemma}
\begin{proof}
	By definition:
	\begin{align*}
	\hspace{1em}&\hspace{-1em} \mathscr B_1 (\bm q_h,y_h,\widehat y_h^o;\bm p_h,-z_h,-\widehat z_h^o) + \mathscr B_2 (\bm p_h,z_h,\widehat z_h^o;-\bm q_h,y_h,\widehat y_h^o)\\
	&=(\bm{q}_h, \bm p_h)_{{\mathcal{T}_h}}- (y_h, \nabla\cdot \bm p_h)_{{\mathcal{T}_h}}+\langle \widehat{y}_h^o, \bm p_h\cdot \bm{n} \rangle_{\partial{{\mathcal{T}_h}}\backslash {\varepsilon_h^{\partial}}} + (\bm{q}_h + \bm{\beta} y_h, \nabla z_h)_{{\mathcal{T}_h}} \\
	&\quad + (\nabla\cdot\bm \beta y_h,  z_h)_{{\mathcal{T}_h}} - \langle\bm q_h\cdot\bm n +h^{-1}P_M y_h+\tau_1 y_h , z_h \rangle_{\partial{{\mathcal{T}_h}}} \\
	&\quad	- \langle(\bm\beta\cdot \bm n-\tau_1-h^{-1})\widehat y_h^o, z_h \rangle_{\partial{{\mathcal{T}_h}}\backslash \varepsilon_h^{\partial}}\\
	&\quad+ \langle\bm q_h\cdot\bm n + \bm{\beta}\cdot\bm n \widehat y_h^o  +h^{-1}(P_My_h - \widehat y_h^o)+\tau_1(y_h-\widehat y_h^o), \widehat z_h^o  \rangle_{\partial{{\mathcal{T}_h}}\backslash\varepsilon_h^{\partial}}\\
	&\quad-(\bm{p}_h, \bm q_h)_{{\mathcal{T}_h}}+ (z_h, \nabla\cdot \bm q_h)_{{\mathcal{T}_h}} -\langle \widehat{z}_h^o, \bm q_h \cdot \bm{n} \rangle_{\partial{{\mathcal{T}_h}}\backslash {\varepsilon_h^{\partial}}}  - (\bm{p}_h - \bm{\beta} z_h, \nabla y_h)_{{\mathcal{T}_h}} \\
	&\quad+\langle\bm p_h\cdot\bm n +h^{-1}P_M z_h+\tau_2 z_h , y_h \rangle_{\partial{{\mathcal{T}_h}}} - \langle (\bm{\beta}\cdot \bm n+\tau_2 + h^{-1}) \widehat z_h^o, y_h \rangle_{\partial{{\mathcal{T}_h}}\backslash \varepsilon_h^{\partial}}\\
	&\quad- \langle\bm p_h\cdot\bm n -\bm{\beta} \cdot\bm n\widehat z_h^o +h^{-1}(P_M z_h - \widehat z_h^o) + \tau_2 (z_h-\widehat z_h^o), \widehat y_h^o \rangle_{\partial{{\mathcal{T}_h}}\backslash\varepsilon_h^{\partial}}.
	\end{align*}
	Integration by parts gives
	\begin{align*}
	\mathscr B_1 &(\bm q_h,y_h,\widehat y_h^o;\bm p_h,-z_h,-\widehat z_h^o) + \mathscr B_2 (\bm p_h,z_h,\widehat z_h^o;-\bm q_h,y_h,\widehat y_h^o)\\
	&=\langle (\tau_2 + \bm{\beta}\cdot\bm n-\tau_1) y_h, z_h \rangle_{\partial\mathcal T_h} + \langle (\tau_2 + \bm{\beta}\cdot\bm n-\tau_1) \widehat y_h^o, \widehat z_h^o \rangle_{\partial\mathcal T_h\backslash\varepsilon_h^\partial}.
	\end{align*}
	Condition \textbf{(A1)} completes the proof.
\end{proof}

\begin{proposition}\label{ex_uni}
	There exists a unique solution of the HDG equations \eqref{HDG_full_discrete}.
\end{proposition}
\begin{proof}
	Since the system \eqref{HDG_full_discrete} is finite dimensional, we only need to prove the uniqueness.  Therefore, we assume $y_d = f = g=0$ and show the system \eqref{HDG_full_discrete} only has the zero solution.
	
	First, take $(\bm r_1,w_1,\mu_1) = (\bm p_h,-z_h,-\widehat z_h^o)$, $(\bm r_2,w_2,\mu_2) = (-\bm q_h,y_h,\widehat y_h^o)$,  and $w_3 = z_h-\gamma u_h $ in the HDG equations \eqref{HDG_full_discrete_a},  \eqref{HDG_full_discrete_b}, and \eqref{HDG_full_discrete_e}, respectively, and sum to obtain
	\begin{align*}
	\hspace{4em}&\hspace{-4em}  \mathscr B_1  (\bm q_h,y_h,\widehat y_h^o;\bm p_h,-z_h,-\widehat z_h^o) + \mathscr B_2 (\bm p_h,z_h,\widehat z_h^o;-\bm q_h,y_h,\widehat y_h^o)\\
	&  = \gamma (y_h,y_h)_{\mathcal T_h} +   (z_h,z_h)_{\mathcal T_h}.
	\end{align*}
	Since $\gamma>0$, Lemma \ref{identical_equa} gives $y_h =  u_h = z_h= 0$.
	
	Next, take $(\bm r_1,w_1,\mu_1) = (\bm q_h,y_h,\widehat y_h^o)$ and $(\bm r_2,w_2,\mu_2) = (\bm p_h,z_h,\widehat z_h^o)$ in Lemma \ref{property_B}, and then use \textbf{(A2)} and \eqref{eqn:tau1_condition} to get $\bm q_h= \bm p_h= \bm 0 $, $ \widehat y_h^o  = \widehat z_h^o=0$.
\end{proof}

\subsection{Proof of the main result}

We follow the proof strategy used in our earlier works \cite{HuShenSinglerZhangZheng_HDG_Dirichlet_control1,HuShenSinglerZhangZheng_HDG_Dirichlet_control5}, and split the proof of the main result into eight steps.  We consider the following auxiliary problem: find $$({\bm{q}}_h(u),{\bm{p}}_h(u), y_h(u), z_h(u), {\widehat{y}}_h^o(u), {\widehat{z}}_h^o(u))\in \bm{V}_h\times\bm{V}_h\times W_h \times W_h\times M_h(o)\times M_h(o)$$ such that
\begin{subequations}\label{HDG_inter_u}
	\begin{align}
	\mathscr B_1(\bm q_h(u),y_h(u),\widehat{y}_h(u);\bm r_1, w_1,\mu_1)&=( f+ u, w_1)_{\mathcal T_h} - \langle P_Mg,\bm r_1\cdot\bm n\rangle \nonumber \\
	&\quad-\langle (\bm\beta\cdot\bm n-h^{-1}-\tau_1) P_Mg,w_1\rangle_{\varepsilon_h^\partial},\label{HDG_u_a}\\
	\mathscr B_2(\bm p_h(u),z_h(u),\widehat{z}_h(u);\bm r_2, w_2,\mu_2)&=(y_d-y_h(u), w_2)_{\mathcal T_h},\label{HDG_u_b}
	\end{align}
\end{subequations}
for all $\left(\bm{r}_1, \bm{r}_2,w_1,w_2,\mu_1,\mu_2\right)\in \bm{V}_h\times\bm{V}_h \times W_h\times W_h\times M_h(o)\times M_h(o)$.

In the first three steps of the proof, we bound the error between the solution components $ (y_h(u), \bm q_h(u) ) $ of part 1 of the auxiliary problem and $ (y,\bm q) $ of the mixed form of the optimality system.  Since $ u $ is the exact optimal control in both problems and is fixed, the source terms in both problems are the same.  We would use the results from \cite{MR3440284} to obtain the error bounds; however, the authors of \cite{MR3440284} pointed us to an error in their work in the $ k = 0 $ case.  To be complete, we present most of the proofs in Steps 1--3, and we use many proof strategies from \cite{MR3440284} in those steps. 


\subsubsection{Step 1: The error equation for part 1 of the auxiliary problem \eqref{HDG_u_a}.}

Define
\begin{equation}\label{notation_1}
\begin{split}
\delta^{\bm q} &=\bm q-{\bm\Pi}\bm q,  \qquad\qquad\qquad \qquad\qquad\qquad\;\;\;\;\varepsilon^{\bm q}_h={\bm\Pi} \bm q-\bm q_h(u),\\
\delta^y&=y- {\Pi} y, \qquad\qquad\qquad \qquad\qquad\qquad\;\;\;\; \;\varepsilon^{y}_h={\Pi} y-y_h(u),\\
\delta^{\widehat y} &= y-P_My,  \qquad\qquad\qquad\qquad\qquad\qquad \;\;\; \varepsilon^{\widehat y}_h=P_M y-\widehat{y}_h(u),\\
\widehat {\bm\delta}_1 &= \delta^{\bm q}\cdot\bm n+h^{-1} P_M \delta^y + \bm{\beta}\cdot\bm n \delta^{\widehat y} + \tau_1(\delta^y - \delta^{\widehat y}).
\end{split}
\end{equation}
where $\widehat y_h(u) = \widehat y_h^o(u)$ on $\varepsilon_h^o$ and $\widehat y_h(u) = P_M g$ on $\varepsilon_h^{\partial}$. This gives $\varepsilon_h^{\widehat y} = 0$ on $\varepsilon_h^{\partial}$.

\begin{lemma}\label{lemma:step1_first_lemma}
	We have
	\begin{align}\label{error_equation_L2k1}
	\mathscr B_1 (\varepsilon_h^{\bm q},\varepsilon_h^{ y}, \varepsilon_h^{\widehat y}, \bm r_1, w_1, \mu_1) &= ( \bm \beta \delta^y, \nabla w_1)_{{\mathcal{T}_h}}
	+(\nabla\cdot\bm{\beta}\delta^y,w_1)_{\mathcal T_h}\nonumber\\
	& \quad  - \langle \widehat{\bm \delta}_1, w_1 \rangle_{\partial{{\mathcal{T}_h}}} + \langle \widehat{\bm \delta}_1, \mu_1 \rangle_{\partial{{\mathcal{T}_h}}\backslash \varepsilon_h^{\partial}}.
	\end{align}
\end{lemma}
\begin{proof}
	By definition:
	\begin{align*}
	\hspace{1em}&\hspace{-1em} 	\mathscr B_1 (\bm \Pi {\bm q},\Pi { y}, P_M  y, \bm r_1, w_1, \mu_1)\\
	&= (\bm \Pi {\bm q}, \bm{r_1})_{{\mathcal{T}_h}}- (\Pi { y}, \nabla\cdot \bm{r_1})_{{\mathcal{T}_h}}+\langle P_M  y, \bm{r_1}\cdot \bm{n} \rangle_{\partial{{\mathcal{T}_h}}\backslash {\varepsilon_h^{\partial}}}  - (\bm \Pi {\bm q} + \bm{\beta} \Pi y, \nabla w_1)_{{\mathcal{T}_h}} \\
	&\quad - (\nabla\cdot\bm{\beta} \Pi y,  w_1)_{{\mathcal{T}_h}} +\langle \bm \Pi {\bm q}\cdot\bm n +h^{-1} P_M\Pi { y}+\tau_1 \Pi y , w_1 \rangle_{\partial{{\mathcal{T}_h}}} \\
	&\quad+ \langle (\bm{\beta}\cdot\bm n-h^{-1}-\tau_1) P_M  y, w_1 \rangle_{\partial{{\mathcal{T}_h}}\backslash \varepsilon_h^{\partial}}\\ 
	&	\quad -\langle  \bm \Pi \bm q\cdot\bm n+\bm \beta\cdot\bm n P_M y +h^{-1}(P_M\Pi y-P_M y) + \tau_1(\Pi y - P_M y),\mu_1\rangle_{\partial\mathcal T_h\backslash\varepsilon^{\partial}_h}.
	\end{align*}
	Properties of the $ L^2 $ projections \eqref{L2_projection} give
	\begin{align*}
	\hspace{1em}&\hspace{-1em}  	\mathscr B_1 (\bm \Pi {\bm q},\Pi { y},P_M  y, \bm r_1, w_1, \mu_1)\\
	& = ( {\bm q}, \bm{r_1})_{{\mathcal{T}_h}}- ({ y}, \nabla\cdot \bm{r_1})_{{\mathcal{T}_h}}+\langle   y, \bm{r_1}\cdot \bm{n} \rangle_{\partial{{\mathcal{T}_h}}\backslash {\varepsilon_h^{\partial}}}\\
	& \quad - ( {\bm q} + \bm \beta y, \nabla w_1)_{{\mathcal{T}_h}} +  (  \bm \beta \delta^y, \nabla w_1)_{{\mathcal{T}_h}} - (\nabla\cdot\bm{\beta} y, w_1)_{\mathcal T_h} + (\nabla\cdot\bm{\beta} \delta^y, w_1)_{\mathcal T_h} \\
	&\quad+\langle {\bm q}\cdot\bm n, w_1 \rangle_{\partial{{\mathcal{T}_h}}} - \langle \delta^{\bm q}\cdot\bm n, w_1 \rangle_{\partial{{\mathcal{T}_h}}}+\langle h^{-1} P_M\Pi { y}+\tau_1\Pi y , w_1 \rangle_{\partial{{\mathcal{T}_h}}}  \\
	&\quad+\langle\bm\beta\cdot\bm n y, w_1\rangle_{\partial\mathcal T_h\backslash\varepsilon_h^\partial} - \langle\bm\beta\cdot\bm n \delta^{\widehat y}, w_1\rangle_{\partial\mathcal T_h\backslash\varepsilon_h^\partial}
	- \langle (h^{-1}+\tau_1) P_M  y, w_1 \rangle_{\partial{{\mathcal{T}_h}}\backslash \varepsilon_h^{\partial}} \\
	&\quad- \langle {\bm q}\cdot\bm n, \mu_1 \rangle_{\partial{{\mathcal{T}_h}}\backslash\varepsilon_h^{\partial}} + \langle \delta^{\bm q}\cdot\bm n, \mu_1 \rangle_{\partial{{\mathcal{T}_h}}\backslash\varepsilon_h^{\partial}} - \langle {\bm \beta}\cdot\bm n y, \mu_1 \rangle_{\partial{{\mathcal{T}_h}}\backslash\varepsilon_h^{\partial}} \\
	&\quad+\langle {\bm \beta}\cdot\bm n \delta^{\widehat y}, \mu_1 \rangle_{\partial{{\mathcal{T}_h}}\backslash\varepsilon_h^{\partial}}+\langle h^{-1} P_M \delta^y, \mu_1 \rangle_{\partial{{\mathcal{T}_h}}\backslash\varepsilon_h^{\partial}} + \langle \tau_1 (\delta^y-\delta^{\widehat y}), \mu_1 \rangle_{\partial{{\mathcal{T}_h}}\backslash\varepsilon_h^{\partial}}.
	\end{align*}
	
	The exact state $ y $ and flux $\bm{q}$ satisfy
	\begin{align*}
	(\bm{q},\bm{r}_1)_{\mathcal{T}_h}-(y,\nabla\cdot \bm{r}_1)_{\mathcal{T}_h}+\left\langle{y},\bm r_1\cdot \bm n \right\rangle_{\partial {\mathcal{T}_h}\backslash\varepsilon_h^\partial} &=- \left\langle g,\bm r_1\cdot \bm n \right\rangle_{\varepsilon_h^\partial},\\
	-(\bm{q}+\bm{\beta} y,\nabla w_1)_{\mathcal{T}_h}-(\nabla \cdot \bm{\beta} y, w_1)_{\mathcal{T}_h}  & \\
	+\left\langle \bm q\cdot \bm n ,w_1\right\rangle_{\partial {\mathcal{T}_h}}+ \left\langle \bm \beta \cdot \bm{n} y,w_1\right\rangle_{\partial {\mathcal{T}_h}\backslash\varepsilon_h^\partial}  &= 	-\left\langle \bm \beta \cdot \bm{n} g,w_1\right\rangle_{\varepsilon_h^\partial} + (f + u,w_1)_{\mathcal{T}_h},\\
	\left\langle ({\bm{q}}+\bm{\beta} y)\cdot \bm{n},\mu_1\right\rangle_{\partial {\mathcal{T}_h}\backslash \varepsilon_h^{\partial}}&=0,
	\end{align*}
	for all $(\bm{r}_1,w_1,\mu_1)\in\bm{V}_h\times W_h\times M_h(o)$. Therefore,
	\begin{align*}
	\hspace{1em}&\hspace{-1em}   \mathscr B_1 (\bm \Pi {\bm q},\Pi { y}, P_M  y, \bm r_1, w_1, \mu_1) \\
	&=-\left\langle g,\bm r_1\cdot \bm n \right\rangle_{\varepsilon_h^{\partial}} - \left\langle \bm{\beta}\cdot \bm n g,w_1\right\rangle_{\varepsilon_h^{\partial}} + (f + u,w_1)_{\mathcal T_h} +  ( \bm \beta \delta^y, \nabla w_1)_{{\mathcal{T}_h}} \\
	&\quad+(\nabla\cdot\bm{\beta}\delta^y,w_1)_{\mathcal T_h} - \langle \delta^{\bm q}\cdot\bm n, w_1 \rangle_{\partial{{\mathcal{T}_h}}}+\langle h^{-1} P_M\Pi { y} +\tau_1 \Pi y, w_1 \rangle_{\partial{{\mathcal{T}_h}}} \\
	&\quad- \langle {\bm \beta}\cdot\bm n \delta^{\widehat y}, w_1 \rangle_{\partial{{\mathcal{T}_h}}\backslash\varepsilon_h^{\partial}}  - \langle (h^{-1}+\tau_1) P_M  y, w_1 \rangle_{\partial{{\mathcal{T}_h}}\backslash \varepsilon_h^{\partial}} + \langle \delta^{\bm q}\cdot\bm n, \mu_1 \rangle_{\partial{{\mathcal{T}_h}}\backslash\varepsilon_h^{\partial}}\\
	&\quad+\langle {\bm \beta}\cdot\bm n \delta^{\widehat y}, \mu_1 \rangle_{\partial{{\mathcal{T}_h}}\backslash\varepsilon_h^{\partial}}+ \langle h^{-1} P_M\delta^y, \mu_1 \rangle_{\partial{{\mathcal{T}_h}}\backslash\varepsilon_h^{\partial}}+ \langle \tau_1 (\delta^y-\delta^{\widehat y}), \mu_1 \rangle_{\partial{{\mathcal{T}_h}}\backslash\varepsilon_h^{\partial}}.
	\end{align*}
	Subtracting part 1 of the auxiliary problem \eqref{HDG_u_a} from the above equality gives the result:
	\begin{align*}
	\hspace{1em}&\hspace{-1em}  \mathscr B_1 (\varepsilon_h^{\bm q},\varepsilon_h^{ y}, \varepsilon_h^{\widehat y},\bm r_1, w_1, \mu_1) \\
	& =  ( \bm \beta \delta^y, \nabla w_1)_{{\mathcal{T}_h}}
	+(\nabla\cdot\bm{\beta}\delta^y,w_1)_{\mathcal T_h} - \langle \delta^{\bm q}\cdot\bm n, w_1 \rangle_{\partial{{\mathcal{T}_h}}} +\langle h^{-1} P_M\Pi { y}, w_1 \rangle_{\partial{{\mathcal{T}_h}}}\\
	&  \quad+\langle \tau_1 \Pi y, w_1 \rangle_{\partial{{\mathcal{T}_h}}}- \langle {\bm \beta}\cdot\bm n \delta^{\widehat y}, w_1 \rangle_{\partial{{\mathcal{T}_h}}}  - \langle (h^{-1}+\tau_1) P_M  y, w_1 \rangle_{\partial{{\mathcal{T}_h}}}\\
	&\quad+ \langle \delta^{\bm q}\cdot\bm n, \mu_1 \rangle_{\partial{{\mathcal{T}_h}}\backslash\varepsilon_h^{\partial}}+\langle {\bm \beta}\cdot\bm n \delta^{\widehat y}, \mu_1 \rangle_{\partial{{\mathcal{T}_h}}\backslash\varepsilon_h^{\partial}}+ \langle h^{-1} P_M\delta^y, \mu_1 \rangle_{\partial{{\mathcal{T}_h}}\backslash\varepsilon_h^{\partial}}\\
	&\quad+ \langle \tau_1 (\delta^y-\delta^{\widehat y}), \mu_1 \rangle_{\partial{{\mathcal{T}_h}}\backslash\varepsilon_h^{\partial}}\\
	&= ( \bm \beta \delta^y, \nabla w_1)_{{\mathcal{T}_h}}
	+(\nabla\cdot\bm{\beta}\delta^y,w_1)_{\mathcal T_h} - \langle \widehat{\bm \delta}_1, w_1 \rangle_{\partial{{\mathcal{T}_h}}} + \langle \widehat{\bm \delta}_1, \mu_1 \rangle_{\partial{{\mathcal{T}_h}}\backslash \varepsilon_h^{\partial}}.
	\end{align*}
\end{proof}

\subsubsection{Step 2: Estimate for $\varepsilon_h^{\boldmath q}$.}  The following key inequality is found in \cite{MR3440284}.
\begin{lemma}\label{nabla_ine}
	We have
	\begin{equation*}
	\|\nabla\varepsilon^y_h\|_{\mathcal T_h}+ h^{-\frac{1}{2}} \| {\varepsilon_h^y -\varepsilon_h^{\widehat y}}\|_{\partial \mathcal T_h}
	\lesssim \|\varepsilon^{\bm q}_h\|_{\mathcal T_h}+h^{-\frac1 2}\|P_M\varepsilon^y_h-\varepsilon^{\widehat y}_h\|_{\partial\mathcal T_h}.
	\end{equation*}
\end{lemma}

\begin{lemma}\label{lemma:step2_main_lemma}
	We have
	\begin{align}
	\norm{\varepsilon_h^{\bm{q}}}_{\mathcal{T}_h}+h^{-\frac 1 2}\|{P_M\varepsilon_h^y-\varepsilon_h^{\widehat{y}}}\|_{\partial \mathcal T_h} \lesssim  h^{k+1}(\|\bm q\|_{k+1,\Omega} + \|y\|_{k+2,\Omega}).
\end{align}
\end{lemma}

\begin{proof}
	First, since $\varepsilon_h^{\widehat y}=0$ on $\varepsilon_h^\partial$, the energy identity for $ \mathscr B_1 $ in Lemma \ref{property_B} gives
	\begin{align*}
	\hspace{2em}&\hspace{-2em}  \mathscr B(\varepsilon_h^{\bm q},\varepsilon_h^{ y}, \varepsilon_h^{\widehat y}, \varepsilon_h^{\bm q},\varepsilon_h^{ y}, \varepsilon_h^{\widehat y})\\ &=(\varepsilon_h^{\bm{q}},\varepsilon_h^{\bm{q}})_{\mathcal{T}_h}+h^{-1} \|{P_M\varepsilon_h^y-\varepsilon_h^{\widehat{y}}}\|_{\partial \mathcal T_h}^2 +\frac 1 2\| (-\nabla\cdot\bm{\beta})^{\frac  1 2}\varepsilon_h^y\|_{\mathcal T_h}^2\\
	& \quad + \|(\tau_1-\frac 12 \bm{\beta} \cdot\bm n)^{\frac 12 } (\varepsilon_h^y-\varepsilon_h^{\widehat{y}})\|_{\partial\mathcal T_h}^2.
	\end{align*}
	Taking $(\bm r_1, w_1,\mu_1) = (\bm \varepsilon_h^{\bm q},\varepsilon_h^y,\varepsilon_h^{\widehat y})$ in \eqref{error_equation_L2k1} in Lemma \ref{lemma:step1_first_lemma} gives
	\begin{equation}\label{step2_2}
	\begin{split}
	(\varepsilon_h^{\bm{q}},\varepsilon_h^{\bm{q}})_{\mathcal{T}_h}&+h^{-1} \|{P_M\varepsilon_h^y-\varepsilon_h^{\widehat{y}}}\|_{\partial \mathcal T_h}^2 +\frac 12\| (-\nabla\cdot\bm{\beta})^{\frac  1 2}\varepsilon_h^y\|_{\mathcal T_h}^2\\
	&\le  ( \bm \beta \delta^y, \nabla \varepsilon_h^y)_{{\mathcal{T}_h}}
	+(\nabla\cdot\bm{\beta}\delta^y,\varepsilon_h^y)_{\mathcal T_h} -\langle \widehat {\bm\delta}_1,\varepsilon_h^y - \varepsilon_h^{\widehat y}\rangle_{\partial\mathcal T_h} \\
	&=: T_1 + T_2 + T_3.
	\end{split}
	\end{equation}
	For the terms $T_1$ and $T_2$, apply Lemma \ref{nabla_ine} and Young's inequality to give
	\begin{align*}
	T_1 &=  ( \bm \beta \delta^y, \nabla \varepsilon_h^y)_{{\mathcal{T}_h}} \le C\|\bm{\beta}\|_{0,\infty,\Omega}^2 \| \delta^y\|_{\mathcal T_h}^2 + \frac 1 4
	\|\varepsilon_h^{\bm{q}}\|_{\mathcal T_h}^2 + \frac 1 {4h} \|{P_M\varepsilon_h^y-\varepsilon_h^{\widehat{y}}}\|_{\partial \mathcal T_h}^2,\\
	T_2 &= (\nabla\cdot\bm{\beta}\delta^y,\varepsilon_h^y)_{\mathcal T_h} \le C \|\delta^y\|_{\mathcal T_h}^2 + \frac 12\|(-\nabla\cdot\bm{\beta})^{\frac 1 2} \varepsilon_h^y\|_{\mathcal T_h}^2.
	\end{align*}
	For the term $T_3$,
	\begin{align*}
	T_3 &= - \langle \widehat {\bm\delta}_1,\varepsilon_h^y - \varepsilon_h^{\widehat y}\rangle_{\partial\mathcal T_h} \\
	&= -\langle \delta^{\bm q}\cdot\bm n+h^{-1} P_M \delta^y + \bm{\beta}\cdot\bm n \delta^{\widehat y} + \tau_1(\delta^y - \delta^{\widehat y}), \varepsilon_h^y - \varepsilon_h^{\widehat y}\rangle_{\partial\mathcal T_h} \\
	& = -\langle \delta^{\bm q}\cdot\bm n + \bm{\beta}\cdot\bm n \delta^{\widehat y} + \tau_1(\delta^y - \delta^{\widehat y}), \varepsilon_h^y - \varepsilon_h^{\widehat y}\rangle_{\partial\mathcal T_h}  -\langle h^{-1} P_M \delta^y, \varepsilon_h^y - \varepsilon_h^{\widehat y}\rangle_{\partial\mathcal T_h}\\
	& =: T_4+T_5.
	\end{align*}
	Applying Lemma \ref{nabla_ine} and Young's inequality again gives
	\begin{align*}
	T_4  &= -\langle \delta^{\bm q}\cdot\bm n + \bm{\beta}\cdot\bm n \delta^{\widehat y} + \tau_1(\delta^y - \delta^{\widehat y}), \varepsilon_h^y - \varepsilon_h^{\widehat y}\rangle_{\partial\mathcal T_h} \\
	&\le C\| h^{1/2}(\delta^{\bm q}\cdot\bm n + \bm{\beta}\cdot\bm n \delta^{\widehat y} + \tau_1(\delta^y - \delta^{\widehat y}))\|_{\partial\mathcal T_h}^2 + \frac 1 C\| h^{-1/2} (\varepsilon_h^y - \varepsilon_h^{\widehat y})\|_{\partial\mathcal T_h}^2\\
	&\le C\| h^{1/2}(\delta^{\bm q}\cdot\bm n + \bm{\beta}\cdot\bm n \delta^{\widehat y} + \tau_1(\delta^y - \delta^{\widehat y}))\|_{\partial\mathcal T_h}^2  + \frac 1 4
	\|\varepsilon_h^{\bm{q}}\|_{\mathcal T_h}^2\\
	&\quad + \frac 1 {4h} \|{P_M\varepsilon_h^y-\varepsilon_h^{\widehat{y}}}\|_{\partial \mathcal T_h}^2.
	\end{align*}
	Finally, for the term $T_5$, we have
	\begin{align*}
	T_5 &= -\langle h^{-1} P_M \delta^y, \varepsilon_h^y - \varepsilon_h^{\widehat y}\rangle_{\partial\mathcal T_h}  = \langle h^{-1}  \delta^y, P_M\varepsilon_h^y - \varepsilon_h^{\widehat y}\rangle_{\partial\mathcal T_h} \\
	&\le  4\|h^{-1/2}\delta^y\|_{\partial\mathcal T_h}^2 +  \frac 1 {4h} \|{P_M\varepsilon_h^y-\varepsilon_h^{\widehat{y}}}\|_{\partial \mathcal T_h}^2.
	\end{align*}
	Sum all the estimates for $\{T_i\}_{i=1}^5$ to obtain
	\begin{align*}
	\hspace{3em}&\hspace{-3em} \|\varepsilon_h^{\bm{q}}\|_{\mathcal{T}_h}^2+h^{-1}\|{P_M\varepsilon_h^y-\varepsilon_h^{\widehat{y}}}\|_{\partial \mathcal T_h}^2 \\
	& \lesssim h\norm{\delta^{\bm q}}_{\partial\mathcal T_h}^2 + h^{-1}
	\norm{\delta^{y}}_{\partial\mathcal T_h}^2 + h
	\|{\delta^{\widehat y}}\|_{\partial\mathcal T_h}^2,\\
	&\lesssim h^{2k+2}(\|\bm q\|_{k+1,\Omega}^2 + \|y\|_{k+2,\Omega}^2).
	\end{align*}
\end{proof}

\subsubsection{Step 3: Estimate for $\varepsilon_h^{y}$ by a duality argument.} 
\label{subsec:proof_step3} 

Next, for any given $\Theta$ in $L^2(\Omega)$ the dual problem is given by
\begin{equation}\label{Dual_PDE}
\begin{split}
\bm\Phi-\nabla\Psi&=0\qquad~\text{in}\ \Omega,\\
\nabla\cdot\bm{\Phi}+\nabla\cdot(\bm\beta\Psi)&=\Theta \qquad\text{in}\ \Omega,\\
\Psi&=0\qquad~\text{on}\ \partial\Omega.
\end{split}
\end{equation}
Since the domain $\Omega$ is convex, we have the following regularity estimate
\begin{align}
\norm{\bm \Phi}_{1,\Omega} + \norm{\Psi}_{2,\Omega} \le C_{\text{reg}} \norm{\Theta}_\Omega,
\end{align}
We use the following quantities in the proof below to estimate $\varepsilon_h^y$:
\begin{align}
\delta^{\bm \Phi} &=\bm \Phi-{\bm\Pi} \bm \Phi, \quad \delta^\Psi=\Psi- {\Pi} \Psi, \quad
\delta^{\widehat \Psi} = \Psi-P_M\Psi.
\end{align}
\begin{lemma}\label{e_sec}
	We have
	\begin{align*}
	\|\varepsilon_h^y\|_{\mathcal T_h} \lesssim  h^{k+1+\min\{k,1\}} (\|\bm q\|_{k+1,\Omega} + \|y\|_{k+2,\Omega}).
	\end{align*}
\end{lemma}

\begin{proof}
	Consider the dual problem \eqref{Dual_PDE} and let $\Theta =- \varepsilon_h^y$.  Take  $(\bm r_1,w_1,\mu_1) = ( {\bm\Pi}\bm{\Phi},{\Pi}\Psi,P_M\Psi)$ in \eqref{error_equation_L2k1} in Lemma \ref{lemma:step1_first_lemma},   and since $\Psi=0$ on $\varepsilon_h^{\partial}$, we have
	\begin{align*}
	\hspace{1em}&\hspace{-1em} \mathscr B_1 (\varepsilon^{\bm q}_h,\varepsilon^y_h,\varepsilon^{\widehat y}_h;{\bm\Pi}\bm{\Phi},{\Pi}\Psi,P_M\Psi) \\
	&= (\varepsilon^{\bm q}_h,{\bm\Pi}\bm{\Phi})_{\mathcal T_h}-( \varepsilon^y_h,\nabla\cdot{\bm\Pi}\bm{\Phi})_{\mathcal T_h}+\langle  \varepsilon^{\widehat y}_h,{\bm\Pi}\bm{\Phi}\cdot\bm n\rangle_{\partial\mathcal T_h\backslash \varepsilon_h^\partial}-(\varepsilon^{\bm q}_h+\bm \beta\varepsilon^y_h,  \nabla {\Pi}\Psi)_{\mathcal T_h}\\
	&\quad-(\nabla\cdot\bm\beta \varepsilon^y_h,{\Pi}\Psi)_{\mathcal T_h}+\langle \varepsilon^{\bm q}_h\cdot\bm n +h^{-1}P_M \varepsilon^y_h +\tau_1 \varepsilon^y_h ,{\Pi}\Psi\rangle_{\partial\mathcal T_h}\\
	&\quad+\langle (\bm\beta\cdot\bm n -h^{-1}-\tau_1) \varepsilon^{\widehat y}_h,{\Pi}\Psi\rangle_{\partial\mathcal T_h\backslash \varepsilon_h^\partial}\\
	&\quad-\langle  \varepsilon^{\bm q}_h\cdot\bm n+\bm \beta\cdot\bm n\varepsilon^{\widehat y}_h +h^{-1}(P_M\varepsilon^y_h-\varepsilon^{\widehat y}_h) + \tau_1(\varepsilon^y_h - \varepsilon^{\widehat y}_h),P_M\Psi\rangle_{\partial\mathcal T_h\backslash\varepsilon^{\partial}_h}\\
	&= (\varepsilon^{\bm q}_h,\bm{\Phi})_{\mathcal T_h}-( \varepsilon^y_h,\nabla\cdot \bm{\Phi})_{\mathcal T_h} + ( \varepsilon^y_h,\nabla\cdot\delta^{\bm{\Phi}})_{\mathcal T_h}-\langle  \varepsilon^{\widehat y}_h,\delta^{\bm{\Phi}}\cdot\bm n\rangle_{\partial\mathcal T_h}\\
	&\quad-(\varepsilon^{\bm q}_h+\bm \beta\varepsilon^y_h,  \nabla \Psi)_{\mathcal T_h}+(\varepsilon^{\bm q}_h+\bm \beta\varepsilon^y_h,  \nabla \delta^\Psi)_{\mathcal T_h} - (\nabla\cdot\bm\beta \varepsilon^y_h,\Psi)_{\mathcal T_h} + (\nabla\cdot\bm\beta \varepsilon^y_h,\delta^\Psi)_{\mathcal T_h}\\
	& \quad- \langle  \varepsilon^{\bm q}_h\cdot\bm n+\bm \beta\cdot\bm n\varepsilon^{\widehat y}_h +h^{-1}(P_M\varepsilon^y_h-\varepsilon^{\widehat y}_h) + \tau_1(\varepsilon^y_h - \varepsilon^{\widehat y}_h),\delta^\Psi - \delta^{\widehat\Psi}\rangle_{\partial\mathcal T_h}.
	\end{align*}
	Here we used $\langle\varepsilon^{\widehat y}_h,\bm \Phi\cdot\bm n\rangle_{\partial\mathcal T_h}=0$, which holds since $\varepsilon^{\widehat y}_h$ is a single-valued function on interior edges and $\varepsilon^{\widehat y}_h=0$ on $\varepsilon^{\partial}_h$. 
	
	Next, integration by parts gives
	\begin{equation}\label{inde_eq2}
	\begin{split}
	(\varepsilon^y_h,\nabla\cdot\delta^{\bm \Phi})_{\mathcal{T}_h}
	&=\langle \varepsilon^y_h,\delta^{\bm \Phi} \cdot\bm n\rangle_{\partial\mathcal T_h}-(\nabla\varepsilon^y_h,\delta^{\bm \Phi})_{\mathcal{T}_h} = \langle \varepsilon^y_h,\delta^{\bm \Phi}\cdot\bm n\rangle_{\partial\mathcal T_h},\\
	(\varepsilon^{\bm q}_h, \nabla \delta^{ \Psi})_{\mathcal{T}_h}&=\langle \varepsilon^{\bm q}_h \cdot\bm n, \delta^{ \Psi}\rangle_{\partial\mathcal T_h}-(\nabla\cdot \varepsilon^{\bm q}_h , \delta^{ \Psi})_{\mathcal T_h} = \langle \varepsilon^{\bm q}_h \cdot\bm n, \delta^{ \Psi}\rangle_{\partial\mathcal T_h},\\
	(\bm\beta \varepsilon_h^y, \nabla \delta^{ \Psi})_{\mathcal{T}_h}&=\langle \bm{\beta} \cdot\bm n \varepsilon_h^y, \delta^{ \Psi}\rangle_{\partial\mathcal T_h}- (\nabla\cdot \bm{\beta} \varepsilon_h^y, \delta^{ \Psi})_{\mathcal T_h}  - (\bm{\beta} \nabla\varepsilon_h^y, \delta^{ \Psi})_{\mathcal T_h}. 
	\end{split}
	\end{equation}
	We have
	\begin{align*}
	\hspace{3em}&\hspace{-3em}  \mathscr B_1 (\varepsilon^{\bm q}_h,\varepsilon^y_h,\varepsilon^{\widehat y}_h;{\bm\Pi}\bm{\Phi},{\Pi}\Psi,P_M\Psi)\\
	&=\|  \varepsilon_h^y\|_{\mathcal T_h}^2 + \langle \varepsilon^y_h - \varepsilon^{\widehat y}_h,\delta^{\bm \Phi}\cdot\bm n +\bm{\beta}\cdot\bm n  \delta^{\Psi} \rangle_{\partial\mathcal T_h} - (\nabla \varepsilon_h^y,\bm{\beta}\delta^{\Psi})_{\mathcal T_h}\\
	& \quad- \langle   h^{-1}(P_M\varepsilon^y_h-\varepsilon^{\widehat y}_h) + \tau_1(\varepsilon^y_h-\varepsilon^{\widehat y}_h),\delta^\Psi - \delta^{\widehat\Psi}\rangle_{\partial\mathcal T_h}.
	\end{align*}
	On the other hand, since $\Psi = 0$ on $\varepsilon_h^\partial $ the error equation \eqref{error_equation_L2k1} in Lemma \ref{lemma:step1_first_lemma} gives
	\begin{align*}
	\hspace{3em}&\hspace{-3em} \mathscr B_1 (\varepsilon^{\bm q}_h,\varepsilon^y_h,\varepsilon^{\widehat y}_h;{\bm\Pi}\bm{\Phi},{\Pi}\Psi,P_M\Psi)\\
	&=( \bm \beta \delta^y, \nabla {\Pi}\Psi)_{{\mathcal{T}_h}}
	+(\nabla\cdot\bm{\beta}\delta^y,{\Pi}\Psi)_{\mathcal T_h} + \langle \widehat{\bm \delta}_1,\delta^{\Psi} - \delta^{\widehat \Psi} \rangle_{\partial{{\mathcal{T}_h}}}.
	\end{align*}
	Comparing the above two equalities, we get
	\begin{align*}
	& \|  \varepsilon_h^y\|_{\mathcal T_h}^2\\
	  & \ \ = - \langle \varepsilon^y_h - \varepsilon^{\widehat y}_h,\delta^{\bm \Phi}\cdot\bm n +\bm{\beta}\cdot\bm n  \delta^{\Psi} \rangle_{\partial\mathcal T_h} + (\nabla \varepsilon_h^y,\bm{\beta}\delta^{\Psi})_{{\mathcal{T}_h}}+( \bm \beta \delta^y, \nabla {\Pi}\Psi)_{{\mathcal{T}_h}}\\
	& \qquad +(\nabla\cdot\bm{\beta}\delta^y,{\Pi}\Psi)_{\mathcal T_h}
	+\langle   h^{-1}(P_M\varepsilon^y_h-\varepsilon^{\widehat y}_h) + \tau_1(\varepsilon^y_h-\varepsilon^{\widehat y}_h) + \widehat{\bm \delta}_1,\delta^\Psi - \delta^{\widehat\Psi}\rangle_{\partial\mathcal T_h}\\
	& \ \ =:R_1+R_2+R_3+R_4+R_5.
	\end{align*}
	For the terms $R_1$ and $R_2$,  Lemma \ref{nabla_ine} and Lemma \ref{lemma:step2_main_lemma} give
	\begin{align*}
	R_1 &= - \langle \varepsilon^y_h - \varepsilon^{\widehat y}_h,\delta^{\bm \Phi}\cdot\bm n +\bm{\beta}\cdot\bm n  \delta^{\Psi} \rangle_{\partial\mathcal T_h} \\
	& \le    h^{-\frac 1 2}\|\varepsilon^y_h - \varepsilon^{\widehat y}_h\|_{\partial \mathcal T_h} ~h^{\frac 1 2} \|\delta^{\bm \Phi}\cdot\bm n +\bm{\beta}\cdot\bm n  \delta^{\Psi}\|_{\partial\mathcal T_h}\\
	& \le   h^{-\frac 1 2}\|\varepsilon^y_h - \varepsilon^{\widehat y}_h\|_{\partial \mathcal T_h}  \|\delta^{\bm \Phi}\cdot\bm n +\bm{\beta}\cdot\bm n  \delta^{\Psi}\|_{\mathcal T_h}\\
	& \le  C h^{-\frac 1 2}\|\varepsilon^y_h - \varepsilon^{\widehat y}_h\|_{\partial \mathcal T_h}  (\|\delta^{\bm \Phi}\|_{\mathcal T_h}+ \|  \delta^{\Psi}\|_{\mathcal T_h})\\
	&\le Ch^{k+2}(\|\bm q\|_{k+1,\Omega} + \|y\|_{k+2,\Omega})  \|\varepsilon_h^y\|_{\mathcal T_h},\\
	\bigskip 
	R_2 &=(\nabla \varepsilon_h^y,\bm{\beta}\delta^{\Psi})_{{\mathcal{T}_h}} \le C \|\nabla \varepsilon_h^y\|_{{\mathcal{T}_h}} \|\delta^{\Psi}\|_{{\mathcal{T}_h}}\\
	& \le Ch^{k+2}(\|\bm q\|_{k+1,\Omega} + \|y\|_{k+2,\Omega}) \|\varepsilon_h^y\|_{\mathcal T_h}.
	\end{align*}
	By a simple application of the triangle inequality for the terms $R_3$ and $R_4$, we have 
	\begin{align*}
	R_3 &=( \bm \beta \delta^y, \nabla {\Pi}\Psi)_{{\mathcal{T}_h}}\le  C\|\delta^y\|_{\mathcal T_h} \|\nabla {\Pi}\Psi\|_{\mathcal T_h}\le C\|\delta^y\|_{\mathcal T_h} ( \|\nabla \delta^\Psi\|_{\mathcal T_h} +\|\nabla\Psi\|_{\mathcal T_h})\\
	&\le C\|\delta^y\|_{\mathcal T_h} (h\|\Psi\|_{2,\Omega} + \|\Psi\|_{1,\Omega}  ) \le C\|\delta^y\|_{\mathcal T_h} \|\Psi\|_{2,\Omega}\\
	&\le Ch^{k+2}(\|\bm q\|_{k+1,\Omega} + \|y\|_{k+2,\Omega})\|\varepsilon_h^y\|_{\mathcal T_h},\\
	R_4 &=(\nabla\cdot\bm{\beta}\delta^y,{\Pi}\Psi)_{\mathcal T_h}\le  C\|\delta^y\|_{\mathcal T_h} \| {\Pi}\Psi\|_{\mathcal T_h}\le C\|\delta^y\|_{\mathcal T_h}  ( \| \delta^\Psi\|_{\mathcal T_h} +\|\Psi\|_{\mathcal T_h})\\
	&\le C\|\delta^y\|_{\mathcal T_h} (h^2\|\Psi\|_{2,\Omega} +\|\Psi\|_{\Omega}  ) \le C\|\delta^y\|_{\mathcal T_h} \|\Psi\|_{2,\Omega}\\
	&\le Ch^{k+2}(\|\bm q\|_{k+1,\Omega} + \|y\|_{k+2,\Omega})\|\varepsilon_h^y\|_{\mathcal T_h}.
	\end{align*}
	For the terms $R_1$ to $R_4$, we obtain the optimal convergence rate for $k\ge 0$. However, we only get the optimal convergence rate for $R_5$ when $k\ge 1$.
	\begin{align*}
	R_5 &=\langle   h^{-1}(P_M\varepsilon^y_h-\varepsilon^{\widehat y}_h) + \tau_1(\varepsilon^y_h-\varepsilon^{\widehat y}_h) + \widehat{\bm \delta}_1,\delta^\Psi - \delta^{\widehat\Psi}\rangle_{\partial\mathcal T_h}\\
	& \le  \|h^{-1}(P_M\varepsilon^y_h-\varepsilon^{\widehat y}_h) + \tau_1(\varepsilon^y_h-\varepsilon^{\widehat y}_h) + \widehat{\bm \delta}_1\|_{\partial\mathcal T_h} \|\delta^\Psi - \delta^{\widehat\Psi}\|_{\partial\mathcal T_h}\\
	&\le C(h^{-1}\|(P_M\varepsilon^y_h-\varepsilon^{\widehat y}_h) \|_{\partial\mathcal T_h} + \| \varepsilon^y_h-\varepsilon^{\widehat y}_h\|_{\partial\mathcal T_h}+ \|\widehat{\bm \delta}_1\|_{\partial\mathcal T_h} )\|\delta^\Psi - \delta^{\widehat\Psi}\|_{\partial\mathcal T_h}.
	\end{align*}
	It is straightforward to get
	\begin{align*}
	\hspace{3em}&\hspace{-3em}  h^{-1}\|(P_M\varepsilon^y_h-\varepsilon^{\widehat y}_h) \|_{\partial\mathcal T_h} + \| \varepsilon^y_h-\varepsilon^{\widehat y}_h\|_{\partial\mathcal T_h}+ \|\widehat{\bm \delta}_1\|_{\partial\mathcal T_h}\\
	& \le C h^{k+\frac 1 2} (\|\bm q\|_{k+1,\Omega} + \|y\|_{k+2,\Omega}),
	\end{align*}
	and 
	\begin{align*}
	\|\delta^\Psi - \delta^{\widehat\Psi}\|_{\partial\mathcal T_h} \le C h^{\min\{k,1\}+\frac 1 2}  \|\varepsilon_h^y\|_{\mathcal T_h}.
	\end{align*}
	This gives
	\begin{align*}
	R_5 \le Ch^{k+1+\min\{k,1\}} (\|\bm q\|_{k+1,\Omega} + \|y\|_{k+2,\Omega})  \|\varepsilon_h^y\|_{\mathcal T_h}.
	\end{align*}
	Finally, we complete the proof  by summing the estimates for $R_1$ to $R_5$.
\end{proof}

The triangle inequality gives convergence rates for $\|\bm q -\bm q_h(u)\|_{\mathcal T_h}$ and $\|y -y_h(u)\|_{\mathcal T_h}$:
\begin{lemma}\label{lemma:step3_conv_rates}
	\begin{subequations}
		\begin{align}
		\|\bm q -\bm q_h(u)\|_{\mathcal T_h}&\le \|\delta^{\bm q}\|_{\mathcal T_h} + \|\varepsilon_h^{\bm q}\|_{\mathcal T_h} \nonumber\\
		& \lesssim h^{k+1} (\|\bm q\|_{k+1,\Omega} + \|y\|_{k+2,\Omega}),\label{error_qu}\\
		\|y -y_h(u)\|_{\mathcal T_h}&\le \|\delta^{y}\|_{\mathcal T_h} + \|\varepsilon_h^{y}\|_{\mathcal T_h}\nonumber\\ 
		&\lesssim  h^{k+1+\min\{k,1\}} (\|\bm q\|_{k+1,\Omega} + \|y\|_{k+2,\Omega})\label{error_yu}.
		\end{align}
	\end{subequations}
\end{lemma}

\subsubsection{Step 4: The error equation for part 2 of the auxiliary problem \eqref{HDG_u_b}.}  Next, we consider the dual variables, i.e., the state $ z $ and the flux $ \bm{p} $, and bound the error between the solutions of part 2 of the auxiliary problem and the mixed form \eqref{mixed_a}-\eqref{mixed_d} of the optimality system.  Define
%
%
\begin{equation}\label{notation_3}
\begin{split}
\delta^{\bm p} &=\bm p-{\bm\Pi}\bm p,  \qquad\qquad\qquad \qquad\qquad\qquad\;\;\;\;\varepsilon^{\bm p}_h={\bm\Pi} \bm p-\bm p_h(u),\\
\delta^z&=z- {\Pi} z, \qquad\qquad\qquad \qquad\qquad\qquad\;\;\;\; \;\varepsilon^{z}_h={\Pi} z-z_h(u),\\
\delta^{\widehat z} &= z-P_Mz,  \qquad\qquad\qquad\qquad\qquad\qquad \;\;\; \varepsilon^{\widehat z}_h=P_M z-\widehat{z}_h(u),\\
\widehat {\bm\delta}_2 &= \delta^{\bm p}\cdot\bm n+h^{-1} P_M \delta^z + \bm{\beta}\cdot\bm n \delta^{\widehat z} + \tau_2(\delta^z - \delta^{\widehat z}).
\end{split}
\end{equation}


\begin{lemma}\label{lemma:step4_first_lemma}
	We have
	\begin{align}
	\mathscr B_2 (\varepsilon_h^{\bm p},\varepsilon_h^{z}, \varepsilon_h^{\widehat z}, \bm r_2, w_2, \mu_2) &=( \bm \beta \delta^z, \nabla w_2)_{{\mathcal{T}_h}} - \langle \widehat{\bm \delta}_2, w_2 \rangle_{\partial{{\mathcal{T}_h}}} + \langle \widehat{\bm \delta}_2, \mu_2 \rangle_{\partial{{\mathcal{T}_h}}\backslash \varepsilon_h^{\partial}}\nonumber\\
	&\quad+(y-y_h(u),w_2)_{\mathcal T_h}.\label{error_equation_L2k2}
	\end{align}
\end{lemma}
The proof is similar to the proof of Lemma \ref{lemma:step1_first_lemma} and is omitted.

\subsubsection{Step 5: Estimate for $\varepsilon_h^{\boldmath p}$.}
The following discrete Poincar{\'e} inequality can be found in \cite{MR3440284}.
\begin{lemma}\label{lemma:discr_Poincare_ineq}
	We have
	\begin{align}\label{poin_in}
	\|\varepsilon_h^z\|_{\mathcal T_h} \le C(\|\nabla \varepsilon_h^z\|_{\mathcal T_h} + h^{-\frac 1 2} \|\varepsilon_h^z - \varepsilon_h^{\widehat z}\|_{\partial\mathcal T_h}).
	\end{align}
\end{lemma}

\begin{lemma}\label{lemma:step5_main_lemma}
	We have
	\begin{subequations}
		\begin{align}
	\hspace{3em}&\hspace{-3em} \norm{\varepsilon_h^{\bm{p}}}_{\mathcal{T}_h}+h^{-\frac 1 2}\|{P_M\varepsilon_h^z-\varepsilon_h^{\widehat{z}}}\|_{\partial \mathcal T_h} \nonumber \\
	&\lesssim  h^{k+1}(\|\bm q\|_{k+1,\Omega} + \|y\|_{k+2,\Omega} + \|\bm p\|_{k+1,\Omega} + \|z\|_{k+2,\Omega}),\label{error_p}\\
\norm{\varepsilon_h^{{z}}}_{\mathcal{T}_h}&\lesssim h^{k+1}(\|\bm q\|_{k+1,\Omega} + \|y\|_{k+2,\Omega} + \|\bm p\|_{k+1,\Omega} + \|z\|_{k+2,\Omega}).\label{error_z}
\end{align}
	\end{subequations}
\end{lemma}

\begin{proof}
	First, we note the key inequality in Lemma \ref{nabla_ine} can be applied with $ (z,\bm p, \hat z) $ replaced by $ (y,\bm q, \hat y) $.  This gives
	\begin{align}\label{nabla_z}
	\|\nabla \varepsilon_h^z\|_{\mathcal T_h} + h^{-\frac 1 2} \|\varepsilon_h^z - \varepsilon_h^{\widehat z}\|_{\partial\mathcal T_h} \lesssim \|\varepsilon^{\bm p}_h\|_{\mathcal T_h}+h^{-\frac1 2}\|P_M\varepsilon^z_h-\varepsilon^{\widehat z}_h\|_{\partial\mathcal T_h}.
	\end{align}

	Next, since $\varepsilon_h^{\widehat z}=0$ on $\varepsilon_h^\partial$, the energy identity for $ \mathscr B_2 $ in Lemma \ref{property_B} gives
	\begin{align*}
	 \hspace{3em}&\hspace{-3em} \mathscr B_2 (\varepsilon_h^{\bm p},\varepsilon_h^{ z}, \varepsilon_h^{\widehat z}, \varepsilon_h^{\bm p},\varepsilon_h^{z}, \varepsilon_h^{\widehat z})\\ &=(\varepsilon_h^{\bm{p}},\varepsilon_h^{\bm{p}})_{\mathcal{T}_h}+h^{-1} \|{P_M\varepsilon_h^z-\varepsilon_h^{\widehat{z}}}\|_{\partial \mathcal T_h}^2+\frac 1 2\| (-\nabla\cdot\bm{\beta})^{\frac  1 2}\varepsilon_h^z\|_{\mathcal T_h}^2\\
	 &\quad + \|(\tau_2+\frac 12 \bm{\beta} \cdot\bm n)^{\frac 12 } (\varepsilon_h^z-\varepsilon_h^{\widehat{z}})\|_{\partial\mathcal T_h}^2.
	 \end{align*}
	Then taking $(\bm r_2, w_2,\mu_2) = (\bm \varepsilon_h^{\bm p},\varepsilon_h^z,\varepsilon_h^{\widehat z})$ in \eqref{error_equation_L2k2} in Lemma \ref{lemma:step4_first_lemma} gives
	\begin{align*}
	(\varepsilon_h^{\bm{p}},\varepsilon_h^{\bm{p}})_{\mathcal{T}_h}&+h^{-1} \|{P_M\varepsilon_h^z-\varepsilon_h^{\widehat{z}}}\|_{\partial \mathcal T_h}^2 + \|(\tau_2+\frac 12 \bm{\beta} \cdot\bm n)^{\frac 12 } (\varepsilon_h^y-\varepsilon_h^{\widehat{y}})\|_{\partial\mathcal T_h}^2\\
	& \leq ( \bm \beta \delta^z, \nabla \varepsilon_h^z)_{{\mathcal{T}_h}}  -\langle \widehat {\bm\delta}_2,\varepsilon_h^z - \varepsilon_h^{\widehat z}\rangle_{\partial\mathcal T_h} + (y-y_h(u),\varepsilon_h^z)_{\mathcal T_h}\\
	& =: T_1+T_2+T_3.
	\end{align*}
	As in the proof of Lemma \ref{lemma:step2_main_lemma}, apply \eqref{nabla_z} and Young's inequality to obtain
	\begin{align*}
	T_1 &=  ( \bm \beta \delta^z, \nabla \varepsilon_h^z)_{{\mathcal{T}_h}}\\
	& \le C \| \delta^z\|_{\mathcal T_h}^2 + \frac 1 4
	\|\varepsilon_h^{\bm{p}}\|_{\mathcal T_h}^2 + \frac 1 {4h} \|{P_M\varepsilon_h^z-\varepsilon_h^{\widehat{z}}}\|_{\partial \mathcal T_h}^2,\\
	T_2 &=  -\langle \widehat {\bm\delta}_2,\varepsilon_h^z - \varepsilon_h^{\widehat z}\rangle_{\partial\mathcal T_h} \\
	&\le C(\|\delta^{\bm p}\|_{\mathcal T_h}^2 +h^{-2}\|\delta^z\|_{\mathcal T_h}^2 + h\|\delta^{\widehat z}\|_{\partial\mathcal T_h}^2 ) \\
	&\quad+ \frac 1 4
	\|\varepsilon_h^{\bm{p}}\|_{\mathcal T_h}^2 + \frac 1 {4h} \|{P_M\varepsilon_h^z-\varepsilon_h^{\widehat{z}}}\|_{\partial \mathcal T_h}^2.
	\end{align*}
	For the term $T_3$, we have
	\begin{align*}
	T_3&= (y-y_h(u),\varepsilon_h^z)_{\mathcal T_h} \le  \|y-y_h(u)\|_{\mathcal T_h} \|\varepsilon_h^z\|_{\mathcal T_h}\\
	&\le C\|y-y_h(u)\|_{\mathcal T_h} (\|\nabla \varepsilon_h^z\|_{\mathcal T_h} + h^{-\frac 1 2} \|\varepsilon_h^z - \varepsilon_h^{\widehat z}\|_{\partial\mathcal T_h})\\
	&\le  C\|y-y_h(u)\|_{\mathcal T_h} (\|\varepsilon^{\bm p}_h\|_{\mathcal T_h}+h^{-\frac1 2}\|P_M\varepsilon^z_h-\varepsilon^{\widehat z}_h\|_{\partial\mathcal T_h})\\
	&\le  C\|y-y_h(u)\|_{\mathcal T_h}^2 + \frac 1 4
	\|\varepsilon_h^{\bm{p}}\|_{\mathcal T_h}^2 + \frac 1 {4h} \|{P_M\varepsilon_h^z-\varepsilon_h^{\widehat{z}}}\|_{\partial \mathcal T_h}^2.
	\end{align*}
	Summing $T_1$ to $T_3$ gives
	\begin{align*}
	\hspace{3em}&\hspace{-3em}  \norm{\varepsilon_h^{\bm{p}}}_{\mathcal{T}_h}+h^{-\frac 1 2}\|{P_M\varepsilon_h^z-\varepsilon_h^{\widehat{z}}}\|_{\partial \mathcal T_h}\\
	 &\le Ch^{k+1}(\|\bm q\|_{k+1,\Omega} + \|y\|_{k+2,\Omega} + \|\bm p\|_{k+1,\Omega} + \|z\|_{k+2,\Omega}).
	\end{align*}
	Finally, \eqref{poin_in}, \eqref{error_p}, and \eqref{nabla_z} together imply \eqref{error_z}.
\end{proof}

\subsubsection{Step 6: Estimate for $\varepsilon_h^{z}$ by a duality argument.} 
\label{subsec:proof_step6} 

For $\Theta$ given in $L^2(\Omega)$, we consider the dual problem for $ z $:
\begin{equation}\label{Dual_PDE2}
\begin{split}
\bm\Phi-\nabla\Psi&=0\qquad~\text{in}\ \Omega,\\
\nabla\cdot\bm{\Phi}-\bm\beta\cdot\nabla\Psi&=\Theta \qquad\text{in}\ \Omega,\\
\Psi&=0\qquad~\text{on}\ \partial\Omega.
\end{split}
\end{equation}
Again since the domain $\Omega$ is convex, we have the regularity estimate
\begin{align}
\norm{\bm \Phi}_{1,\Omega} + \norm{\Psi}_{2,\Omega} \le C_{\text{reg}} \norm{\Theta}_\Omega,
\end{align}
Before we estimate  $\varepsilon_h^z$, we repeat the notation in \eqref{e_sec}:
\begin{align*}
\delta^{\bm \Phi} &=\bm \Phi-{\bm\Pi} \bm \Phi, \quad \delta^\Psi=\Psi- {\Pi} \Psi, \quad
\delta^{\widehat \Psi} = \Psi-P_M\Psi.
\end{align*}
\begin{lemma}\label{e_secz}
	We have
	\begin{align*}
	\|\varepsilon_h^z\|_{\mathcal T_h} \le Ch^{k+1+\min\{k,1\}}(\|\bm q\|_{k+1,\Omega} + \|y\|_{k+2,\Omega} + \|\bm p\|_{k+1,\Omega} + \|z\|_{k+2,\Omega}).
	\end{align*}
\end{lemma}

\begin{proof}
	Consider the dual problem \eqref{Dual_PDE2} and let $\Theta = \varepsilon_h^z$. We take  $(\bm r_2,w_2,\mu_2) = ( {\bm\Pi}\bm{\Phi},{\Pi}\Psi,P_M\Psi)$ in \eqref{error_equation_L2k2} in Lemma \ref{lemma:step4_first_lemma},   and since $\Psi=0$ on $\varepsilon_h^{\partial}$, we have
	\begin{align*}
	\hspace{3em}&\hspace{-3em}  \mathscr B_2 (\varepsilon^{\bm p}_h,\varepsilon^z_h,\varepsilon^{\widehat z}_h;{\bm\Pi}\bm{\Phi},{\Pi}\Psi,P_M\Psi)\\
	&=
	(\varepsilon^{\bm p}_h,{\bm\Pi}\bm{\Phi})_{\mathcal T_h}-( \varepsilon^z_h,\nabla\cdot{\bm\Pi}\bm{\Phi})_{\mathcal T_h}+\langle  \varepsilon^{\widehat z}_h,{\bm\Pi}\bm{\Phi}\cdot\bm n\rangle_{\partial\mathcal T_h\backslash \varepsilon_h^\partial}\\
	&\quad-(\varepsilon^{\bm p}_h-\bm \beta\varepsilon^z_h,  \nabla {\Pi}\Psi)_{\mathcal T_h}+\langle \varepsilon^{\bm p}_h\cdot\bm n +h^{-1}P_M \varepsilon^z_h +\tau_2 \varepsilon^z_h ,{\Pi}\Psi\rangle_{\partial\mathcal T_h}\\
	&\quad-\langle (\bm\beta\cdot\bm n +h^{-1}+\tau_1) \varepsilon^{\widehat y}_h,{\Pi}\Psi\rangle_{\partial\mathcal T_h\backslash \varepsilon_h^\partial}\\
	&\quad-\langle  \varepsilon^{\bm p}_h\cdot\bm n+\bm \beta\cdot\bm n\varepsilon^{\widehat z}_h +h^{-1}(P_M\varepsilon^z_h-\varepsilon^{\widehat z}_h) + \tau_2(\varepsilon^z_h - \varepsilon^{\widehat z}_h),P_M\Psi\rangle_{\partial\mathcal T_h\backslash\varepsilon^{\partial}_h}\\
	&= (\varepsilon^{\bm p}_h,\bm{\Phi})_{\mathcal T_h}-( \varepsilon^z_h,\nabla\cdot \bm{\Phi})_{\mathcal T_h} + ( \varepsilon^z_h,\nabla\cdot\delta^{\bm{\Phi}})_{\mathcal T_h}-\langle  \varepsilon^{\widehat z}_h,\delta^{\bm{\Phi}}\cdot\bm n\rangle_{\partial\mathcal T_h}\\
	&\quad-(\varepsilon^{\bm p}_h-\bm \beta\varepsilon^z_h,  \nabla \Psi)_{\mathcal T_h}+(\varepsilon^{\bm p}_h-\bm \beta\varepsilon^z_h,  \nabla \delta^\Psi)_{\mathcal T_h}\\
	&\quad  - \langle  \varepsilon^{\bm p}_h\cdot\bm n-\bm \beta\cdot\bm n\varepsilon^{\widehat z}_h +h^{-1}(P_M\varepsilon^z_h-\varepsilon^{\widehat z}_h) + \tau_2(\varepsilon^z_h - \varepsilon^{\widehat z}_h),\delta^\Psi - \delta^{\widehat\Psi}\rangle_{\partial\mathcal T_h}.
	\end{align*}
	Here, we have $\langle\varepsilon^{\widehat z}_h,\bm \Phi\cdot\bm n\rangle_{\partial\mathcal T_h}=0$, which holds since $\varepsilon^{\widehat z}_h$ is single-valued function on interior edges and $\varepsilon^{\widehat z}_h=0$ on $\varepsilon^{\partial}_h$. 
	
	The same argument in \eqref{inde_eq2} gives
	\begin{align*}
	(\varepsilon^z_h,\nabla\cdot\delta^{\bm \Phi})_{\mathcal{T}_h}
	&=\langle \varepsilon^z_h,\delta^{\bm \Phi} \cdot\bm n\rangle_{\partial\mathcal T_h}-(\nabla\varepsilon^z_h,\delta^{\bm \Phi})_{\mathcal{T}_h} = \langle \varepsilon^z_h,\delta^{\bm \Phi}\cdot\bm n\rangle_{\partial\mathcal T_h},\\
	(\varepsilon^{\bm p}_h, \nabla \delta^{ \Psi})_{\mathcal{T}_h}&=\langle \varepsilon^{\bm p}_h \cdot\bm n, \delta^{ \Psi}\rangle_{\partial\mathcal T_h}-(\nabla\cdot \varepsilon^{\bm p}_h , \delta^{ \Psi})_{\mathcal T_h} = \langle \varepsilon^{\bm p}_h \cdot\bm n, \delta^{ \Psi}\rangle_{\partial\mathcal T_h},\\
	(\bm\beta \varepsilon_h^z, \nabla \delta^{ \Psi})_{\mathcal{T}_h}&=\langle \bm{\beta} \cdot\bm n \varepsilon_h^z, \delta^{ \Psi}\rangle_{\partial\mathcal T_h}- (\nabla\cdot \bm{\beta} \varepsilon_h^z, \delta^{ \Psi})_{\mathcal T_h}  - (\bm{\beta} \nabla\varepsilon_h^z, \delta^{ \Psi})_{\mathcal T_h}. 
	\end{align*}
	Then, 
	\begin{align*}
	\hspace{3em}&\hspace{-3em}  \mathscr B_2 (\varepsilon^{\bm p}_h,\varepsilon^z_h,\varepsilon^{\widehat z}_h;{\bm\Pi}\bm{\Phi},{\Pi}\Psi,P_M\Psi)\\
	&=\|  \varepsilon_h^z\|_{\mathcal T_h}^2 + \langle \varepsilon^z_h - \varepsilon^{\widehat z}_h,\delta^{\bm \Phi}\cdot\bm n -\bm{\beta}\cdot\bm n  \delta^{\Psi} \rangle_{\partial\mathcal T_h} + (\nabla \varepsilon_h^z,\bm{\beta}\delta^{\Psi})\\
	&\quad+ (\nabla\cdot \bm{\beta} \varepsilon_h^z, \delta^{ \Psi})_{\mathcal T_h} - \langle   h^{-1}(P_M\varepsilon^z_h-\varepsilon^{\widehat z}_h) + \tau_1(\varepsilon^z_h-\varepsilon^{\widehat z}_h),\delta^\Psi - \delta^{\widehat\Psi}\rangle_{\partial\mathcal T_h},
	\end{align*}
	where we have used  $\varepsilon^{\widehat z}_h$ is single-valued function on interior edges and  $\varepsilon^{\widehat z}_h=0$ on $\varepsilon^{\partial}_h$.
	On the other hand, 
	\begin{align*}
	\hspace{3em}&\hspace{-3em}  \mathscr B_2 (\varepsilon^{\bm p}_h,\varepsilon^z_h,\varepsilon^{\widehat z}_h;{\bm\Pi}\bm{\Phi},{\Pi}\Psi,P_M\Psi)\\
	&=( \bm \beta \delta^z, \nabla {\Pi}\Psi)_{{\mathcal{T}_h}} + \langle \widehat{\bm \delta}_1,\delta^{\Psi} - \delta^{\widehat \Psi} \rangle_{\partial{{\mathcal{T}_h}}} +(y-y_h(u),\Pi\Psi)_{\mathcal T_h}.
	\end{align*}
	Comparing the above two equalities gives
	\begin{align*}
	\|  \varepsilon_h^z\|_{\mathcal T_h}^2  &= - \langle \varepsilon^z_h - \varepsilon^{\widehat z}_h,\delta^{\bm \Phi}\cdot\bm n +\bm{\beta}\cdot\bm n  \delta^{\Psi} \rangle_{\partial\mathcal T_h} - (\nabla \varepsilon_h^z,\bm{\beta}\delta^{\Psi})_{{\mathcal{T}_h}}+(\bm{\beta}\delta^z,\nabla{\Pi}\Psi)_{\mathcal T_h}\\
	&\quad+\langle   h^{-1}(P_M\varepsilon^z_h-\varepsilon^{\widehat z}_h) + \tau_2(\varepsilon^z_h-\varepsilon^{\widehat z}_h) + \widehat{\bm \delta}_2,\delta^\Psi - \delta^{\widehat\Psi}\rangle_{\partial\mathcal T_h}\\
	&\quad-( \nabla \cdot\bm \beta \varepsilon_h^z, \delta^\Psi)_{{\mathcal{T}_h}}+(y-y_h(u),\Pi\Psi)_{\mathcal T_h}\\
	&=:S_1+S_2+S_3+S_4+S_5+S_6.
	\end{align*}
	We can estimate $S_1$ to $S_4$ as in the proof of Lemma \ref{e_sec} to get
	\begin{align*}
	\sum_{i=1}^4 S_i \le Ch^{k+1+\min\{k,1\}}(\|\bm q\|_{k+1,\Omega} + \|y\|_{k+2,\Omega} + \|\bm p\|_{k+1,\Omega} + \|z\|_{k+2,\Omega}).
	\end{align*}
	By the estimate for $\varepsilon_h^z$ in \eqref{error_z} in Lemma \ref{lemma:step5_main_lemma}, we have
	\begin{align*}
	S_5 &=-( \nabla \cdot\bm \beta \varepsilon_h^z, \delta^\Psi)_{{\mathcal{T}_h}}\le  C\|\varepsilon_h^z\|_{\mathcal T_h} \|\delta^\Psi\|_{\mathcal T_h}\\
	&\le Ch^{k+2}(\|\bm q\|_{k+1,\Omega} + \|y\|_{k+2,\Omega} + \|\bm p\|_{k+1,\Omega} + \|z\|_{k+2,\Omega}) \|\varepsilon_h^z\|_{\mathcal T_h}.
	\end{align*}
    The estimate of the last term $S_6$ can be easily obtained from \eqref{lemma:step3_conv_rates}:
	\begin{align*}
	S_6 &=(y-y_h(u),\Pi\Psi)_{\mathcal T_h} \le \|y-y_h(u)\|_{\mathcal T_h}  (\|\delta^\Psi\|_{\mathcal T_h}+\|\Psi\|_{\mathcal T_h})\\
	&\le Ch^{k+1+\min\{k,1\}} \|\varepsilon_h^z\|_{\mathcal T_h}.
	\end{align*}
	Finally, we complete the proof  by combining the estimates for $S_1$ to $S_6$.
\end{proof}

The triangle inequality gives convergence rates for $\|\bm p -\bm p_h(u)\|_{\mathcal T_h}$ and $\|z -z_h(u)\|_{\mathcal T_h}$:
\begin{lemma}\label{lemma:step6_conv_rates}
	\begin{subequations}
		\begin{align}
		\|\bm p -\bm p_h(u)\|_{\mathcal T_h}&\le \|\delta^{\bm p}\|_{\mathcal T_h} + \|\varepsilon_h^{\bm p}\|_{\mathcal T_h} \nonumber\\
		&\lesssim h^{k+1}(\|\bm q\|_{k+1,\Omega} + \|y\|_{k+2,\Omega} + \|\bm p\|_{k+1,\Omega} + \|z\|_{k+2,\Omega})\\
		\|z -z_h(u)\|_{\mathcal T_h}&\le \|\delta^{z}\|_{\mathcal T_h} + \|\varepsilon_h^{z}\|_{\mathcal T_h}\nonumber\\
		& \lesssim  h^{k+1+\min\{k,1\}}(\|\bm q\|_{k+1,\Omega} + \|y\|_{k+2,\Omega} + \|\bm p\|_{k+1,\Omega} + \|z\|_{k+2,\Omega}).
		\end{align}
	\end{subequations}
\end{lemma}

\subsubsection{Step 7: Estimate for $\|u-u_h\|_{\mathcal T_h}$, $\norm {y-y_h}_{\mathcal T_h}$ and $\norm {z-z_h}_{\mathcal T_h}$.}


To obtain the main result, we bound the error between the solutions of the auxiliary problem and the HDG problem \eqref{HDG_full_discrete}.  The proofs of the results in Steps 7 and 8 are similar to the proofs of the corresponding results in our earlier work \cite{HuShenSinglerZhangZheng_HDG_Dirichlet_control5}; we include them for completeness.

For the final steps, let
\begin{equation*}
\begin{split}
\zeta_{\bm q} &=\bm q_h(u)-\bm q_h,\quad\zeta_{y} = y_h(u)-y_h,\quad\zeta_{\widehat y} = \widehat y_h(u)-\widehat y_h,\\
\zeta_{\bm p} &=\bm p_h(u)-\bm p_h,\quad\zeta_{z} = z_h(u)-z_h,\quad\zeta_{\widehat z} = \widehat z_h(u)-\widehat z_h.
\end{split}
\end{equation*}
Subtracting the auxiliary problem and the HDG problem gives the error equations
\begin{subequations}\label{eq_yh}
	\begin{align}
	 \mathscr B_1(\zeta_{\bm q},\zeta_y,\zeta_{\widehat y};\bm r_1, w_1,\mu_1)&=(u-u_h,w_1)_{\mathcal T_h}\label{eq_yh_yhu}\\
	\mathscr B_2(\zeta_{\bm p},\zeta_z,\zeta_{\widehat z};\bm r_2, w_2,\mu_2)&=-(\zeta_y, w_2)_{\mathcal T_h}\label{eq_zh_zhu}.
	\end{align}
\end{subequations}
\begin{lemma}
	We have
	\begin{align}\label{eq_uuh_yhuyh}
	\hspace{3em}&\hspace{-3em}  \gamma\|u-u_h\|^2_{\mathcal T_h}+\|y_h(u)-y_h\|^2_{\mathcal T_h}\nonumber\\
	&=( z_h-\gamma u_h,u-u_h)_{\mathcal T_h}-(z_h(u)-\gamma u,u-u_h)_{\mathcal T_h}.
	\end{align}
\end{lemma}
\begin{proof}
	First, we have
	\begin{align*}
	\hspace{3em}&\hspace{-3em}  ( z_h-\gamma u_h,u-u_h)_{\mathcal T_h}-( z_h(u)-\gamma u,u-u_h)_{\mathcal T_h}\\
	&=-(\zeta_{ z},u-u_h)_{\mathcal T_h}+\gamma\|u-u_h\|^2_{\mathcal T_h}.
	\end{align*}
	Next, Lemma \ref{identical_equa} gives
	\begin{align*}
	\mathscr B_1 &(\zeta_{\bm q},\zeta_y,\zeta_{\widehat{y}};\zeta_{\bm p},-\zeta_{z},-\zeta_{\widehat z}) + \mathscr B_2(\zeta_{\bm p},\zeta_z,\zeta_{\widehat z};-\zeta_{\bm q},\zeta_y,\zeta_{\widehat{y}})  = 0.
	\end{align*}
	On the other hand, working from the definitions yields
	\begin{align*}
	\hspace{3em}&\hspace{-3em}  \mathscr B_1 (\zeta_{\bm q},\zeta_y,\zeta_{\widehat{y}};\zeta_{\bm p},-\zeta_{z},-\zeta_{\widehat z}) + \mathscr B_2(\zeta_{\bm p},\zeta_z,\zeta_{\widehat z};-\zeta_{\bm q},\zeta_y,\zeta_{\widehat{y}})\\
	&= - ( u- u_h,\zeta_{ z})_{\mathcal{T}_h}-\|\zeta_{ y}\|^2_{\mathcal{T}_h}.
	\end{align*}
	Comparing the above two equalities gives
	\begin{align*}
	-(u-u_h,\zeta_{ z})_{\mathcal{T}_h}=\|\zeta_{ y}\|^2_{\mathcal{T}_h},
	\end{align*}
	which completes the proof.
\end{proof}

\begin{theorem}
	We have
	\begin{subequations}
	\begin{align}\label{err_yhu_yh}
	\|u-u_h\|_{\mathcal T_h}&\lesssim h^{k+1+\min\{k,1\}}(|\bm q|_{k+1}+|y|_{k+2}+|\bm p|_{k+1}+|z|_{k+2}),\\
    \|y-y_h\|_{\mathcal T_h}&\lesssim h^{k+1+\min\{k,1\}}(|\bm q|_{k+1}+|y|_{k+2}+|\bm p|_{k+1}+|z|_{k+2}),\\
    \|z-z_h\|_{\mathcal T_h}&\lesssim h^{k+1+\min\{k,1\}}(|\bm q|_{k+1}+|y|_{k+2}+|\bm p|_{k+1}+|z|_{k+2}).
\end{align}
	\end{subequations}
\end{theorem}
\begin{proof}
		The continuous and discretized optimality conditions \eqref{eq_adeq_e} and \eqref{HDG_full_discrete_e} give $ \gamma u = z $ and $ \gamma u_h = z_h $.  Use these equations and the previous lemma to obtain
	\begin{align*}
	\hspace{3em}&\hspace{-3em}   \gamma\|u-u_h\|^2_{\mathcal T_h}+\|\zeta_{ y}\|^2_{\mathcal T_h}\\
	&=( z_h-\gamma u_h,u-u_h)_{\mathcal T_h}-( z_h(u)-\gamma u,u-u_h)_{\mathcal T_h}\\
	&=-( z_h(u)- z,u-u_h)_{\mathcal T_h}\\
	&\le \| z_h(u)- z\|_{\mathcal T_h} \|u-u_h\|_{\mathcal T_h}\\
	&\le\frac{1}{2\gamma}\| z_h(u)- z\|^2_{\mathcal T_h}+\frac{\gamma}{2}\|u-u_h\|^2_{\mathcal T_h}.
	\end{align*}
	By Lemma \ref{lemma:step6_conv_rates}, we have
	\begin{align*}
	\|u-u_h\|_{\mathcal T_h}+\|\zeta_{ y}\|_{\mathcal T_h}&\lesssim h^{k+1+\min\{k,1\}}(|\bm q|_{k+1}+|y|_{k+2}+|\bm p|_{k+1}+|z|_{k+2}).
	\end{align*}
	By the triangle inequality and Lemma \ref{lemma:step3_conv_rates} we obtain
	\begin{align*}
	\|y-y_h\|_{\mathcal T_h}&\lesssim h^{k+1+\min\{k,1\}}(|\bm q|_{k+1}+|y|_{k+2}+|\bm p|_{k+1}+|z|_{k+2}).
	\end{align*}
	Finally, $z = \gamma u $ and $z_h = \gamma u_h$ give
	\begin{align*}
	\|z-z_h\|_{\mathcal T_h}&\lesssim h^{k+1+\min\{k,1\}}(|\bm q|_{k+1}+|y|_{k+2}+|\bm p|_{k+1}+|z|_{k+2}).
	\end{align*}
\end{proof}

\subsubsection{Step 8: Estimate for $\|q-q_h\|_{\mathcal T_h}$ and  $\|p-p_h\|_{\mathcal T_h}$.}

\begin{lemma}
	We have
		\begin{subequations}
	\begin{align}\label{err_Lhu_Lh}
\|\zeta_{\bm q}\|_{\mathcal T_h}&\lesssim h^{k+1+\min\{k,1\}}(|\bm q|_{k+1}+|y|_{k+2}+|\bm p|_{k+1}+|z|_{k+2}),\\
\|\zeta_{\bm p}\|_{\mathcal T_h}&\lesssim h^{k+1+\min\{k,1\}}(|\bm q|_{k+1}+|y|_{k+2}+|\bm p|_{k+1}+|z|_{k+2}).
\end{align}
	\end{subequations}
\end{lemma}
\begin{proof}
	By Lemma \ref{property_B} and the error equation \eqref{eq_yh_yhu}, we have
	\begin{align*}
	\|\zeta_{\bm q}\|^2_{\mathcal T_h} &\lesssim  \mathscr B_1(\zeta_{\bm q},\zeta_y,\zeta_{\widehat y};\zeta_{\bm q},\zeta_y,\zeta_{\widehat y})\\
	&=( u- u_h,\zeta_{ y})_{\mathcal T_h}\\
	&\le\| u- u_h\|_{\mathcal T_h}\|\zeta_{ y}\|_{\mathcal T_h}\\
	&\lesssim h^{2k+2 + 2\min\{k,1\}}(|\bm q|_{k+1}+|y|_{k+1}+|\bm p|_{k+1}+|z|_{k+1})^2.
	\end{align*}
	Similarly,  by Lemma \ref{property_B} and the error equation \eqref{eq_zh_zhu}, we have
	\begin{align*}
	\|\zeta_{\bm p}\|^2_{\mathcal T_h} &\lesssim  \mathscr B_2(\zeta_{\bm p},\zeta_z,\zeta_{\widehat z};\zeta_{\bm p},\zeta_z,\zeta_{\widehat z})\\
	&=-(\zeta_{ y},\zeta_{ z})_{\mathcal T_h}\\
	&\le\|\zeta_{y}\|_{\mathcal T_h}\|\zeta_{ z}\|_{\mathcal T_h}\\
	&\lesssim  h^{2k+2 + 2\min\{k,1\}}(|\bm q|_{k+1}+|y|_{k+1}+|\bm p|_{k+1}+|z|_{k+1})^2.
	\end{align*}
\end{proof}

The above lemma along with the triangle inequality, Lemma \ref{lemma:step3_conv_rates}, and Lemma \ref{lemma:step6_conv_rates} complete the proof of the main result:
\begin{theorem}
	We have
	\begin{subequations}
		\begin{align}
		\|\bm q-\bm q_h\|_{\mathcal T_h}&\lesssim h^{k+1}(|\bm q|_{k+1}+|y|_{k+2}+|\bm p|_{k+1}+|z|_{k+2}),\label{err_q}\\
		\|\bm p-\bm p_h\|_{\mathcal T_h}&\lesssim h^{k+1}(|\bm q|_{k+1}+|y|_{k+2}+|\bm p|_{k+1}+|z|_{k+2})\label{err_p}.
		\end{align}
	\end{subequations}
\end{theorem}

\section{Numerical Experiments}
\label{sec:numerics}

To illustrate our convergence results, we consider two examples on a square domain $\Omega = [0,1]\times [0,1] \subset \mathbb{R}^2$ from our previous work \cite{HuShenSinglerZhangZheng_HDG_Dirichlet_control5}.  We first take $\gamma = 1$ and choose the exact state, dual state, and function $\bm \beta$.  Then we generate the data $f$, $ g $, and $y_d$ using the optimality system \eqref{eq_adeq}.

Table \ref{table_1}--Table \ref{table_4} show the computed errors and convergence rates for $ k = 0 $ and $ k = 1 $ for the two examples.  The computational results match the theory.

\begin{example}\label{example1}
	$\bm{\beta} = [1,1]$, state $ y(x_1,x_2) = \sin(\pi x_1) $, dual state $ z(x_1,x_2) = \sin(\pi x_1)\sin(\pi x_2)$
	\begin{table}
		\begin{center}
			\begin{tabular}{|c|c|c|c|c|c|}
				\hline
				$h/\sqrt 2$ &$1/16$& $1/32$&$1/64$ &$1/128$ & $1/256$ \\
				\hline
				$\norm{\bm{q}-\bm{q}_h}_{0,\Omega}$&1.7274e-01   &9.7054e-02   &5.2507e-02   &2.7509e-02   &1.4111e-02 \\
				\hline
				order&- & 0.83   &0.89   &0.93   &0.96\\
				\hline
				$\norm{\bm{p}-\bm{p}_h}_{0,\Omega}$& 2.5783e-01   &1.4468e-01   &7.7818e-02   &4.0586e-02   &2.0763e-02 \\
				\hline
				order&-&  0.833  &0.89   &0.94   &0.97 \\
				\hline
				$\norm{{y}-{y}_h}_{0,\Omega}$&2.4430e-02   &1.4046e-02   &7.8371e-03   &4.1908e-03   &2.1744e-03\\
				\hline
				order&-& 0.80   &0.84   &0.90   &0.95\\
				\hline
				$\norm{{z}-{z}_h}_{0,\Omega}$& 2.8132e-02   &1.8225e-02   &1.0659e-02   &5.8061e-03   &3.0363e-03 \\
				\hline
				order&-& 0.63   &0.77  &0.88   &0.94 \\
				\hline
			\end{tabular}
		\end{center}
		\caption{Example \ref{example1}: Errors for the state $y$, adjoint state $z$, and the fluxes $\bm q$ and $\bm p$ when $k=0$.}\label{table_1}
	\end{table}

	\begin{table}
		\begin{center}
			\begin{tabular}{|c|c|c|c|c|c|}
				\hline
				$h/\sqrt 2$ &$1/8$& $1/16$&$1/32$ &$1/64$ & $1/128$ \\
				\hline
				$\norm{\bm{q}-\bm{q}_h}_{0,\Omega}$&1.1365e-02   &3.0743e-03   &8.0051e-04   &2.0438e-04   &5.1648e-05 \\
				\hline
				order&-& 1.89   &1.94   &1.97   &1.98\\
				\hline
				$\norm{\bm{p}-\bm{p}_h}_{0,\Omega}$& 2.6923e-02   &6.9736e-03   &1.7764e-03   &4.4849e-04   &1.1269e-04\\
				\hline
				order&-&  1.95&1.97 &1.99 & 2.00 \\
				\hline
				$\norm{{y}-{y}_h}_{0,\Omega}$&1.9986e-03   &2.8351e-04   &3.7918e-05   &4.9101e-06   &6.2497e-07\\
				\hline
				order&-& 2.82  &2.90   &2.95   &2.97 \\
				\hline
				$\norm{{z}-{z}_h}_{0,\Omega}$& 3.8753e-03   &5.3846e-04   &7.1154e-05   &9.1544e-06   &1.1613e-06 \\
				\hline
				order&- & 2.85    &2.92   &2.96   &2.98 \\
				\hline
			\end{tabular}
		\end{center}
		\caption{Example \ref{example1}: Errors for the state $y$, adjoint state $z$, and the fluxes $\bm q$ and $\bm p$ when $k=1$.}\label{table_2}
	\end{table}
\end{example}

\begin{example}\label{example2}
	$\bm{\beta} = [x_2,x_1]$, state $ y(x_1,x_2) = \sin(\pi x_1) $, dual state $ z(x_1,x_2) = \sin(\pi x_1)\sin(\pi x_2)$
	\begin{table}
		\begin{center}
			\begin{tabular}{|c|c|c|c|c|c|}
				\hline
				$h/\sqrt 2$ &$1/16$& $1/32$&$1/64$ &$1/128$ & $1/256$ \\
				\hline
				$\norm{\bm{q}-\bm{q}_h}_{0,\Omega}$ &1.7074e-01   &9.5848e-02   &5.1838e-02   &2.7156e-02   &1.3929e-02 \\
				\hline
				order&-& 0.83   &0.89   &0.93   &0.96\\
				\hline
				$\norm{\bm{p}-\bm{p}_h}_{0,\Omega}$& 2.5679e-01   &1.4404e-01   &7.7454e-02   &4.0391e-02   &2.0661e-02 \\
				\hline
				order&-& 083   &0.90   &0.94   &0.97 \\
				\hline
				$\norm{{y}-{y}_h}_{0,\Omega}$&2.4537e-02   &1.4150e-02   &7.9032e-03   &4.2273e-03   &2.1935e-03\\
				\hline
				order&-&0.79   &0.84   &0.90 &0.95\\ 
				\hline
				$\norm{{z}-{z}_h}_{0,\Omega}$& 2.8293e-02   &1.8369e-02   &1.0747e-02   &5.8549e-03   &3.0618e-03 \\
				\hline
				order&-& 0.62&0.77&0.88& 0.94\\
				\hline
			\end{tabular}
		\end{center}
		\caption{Example \ref{example2}: Errors for the state $y$, adjoint state $z$, and the fluxes $\bm q$ and $\bm p$ when $k=0$.}\label{table_3}
	\end{table}

	\begin{table}
		\begin{center}
			\begin{tabular}{|c|c|c|c|c|c|}
				\hline
				$h/\sqrt 2$ &$1/8$& $1/16$&$1/32$ &$1/64$ & $1/128$ \\
				\hline
				$\norm{\bm{q}-\bm{q}_h}_{0,\Omega}$&1.0144e-02   &2.7469e-03   &7.1555e-04   &1.8271e-04   &4.6174e-05 \\
				\hline
				order&-& 1.88& 1.94  &1.97& 1.98\\
				\hline
				$\norm{\bm{p}-\bm{p}_h}_{0,\Omega}$ & 2.6378e-02   &6.8203e-03   &1.7358e-03   &4.3805e-04   &1.1004e-04 \\
				\hline
				order&-& 1.95&1.97&1.99 & 1.99 \\
				\hline
				$\norm{{y}-{y}_h}_{0,\Omega}$&1.8869e-03   &2.6762e-04   &3.5771e-05   &4.6297e-06   &5.8909e-07\\
				\hline
				order&-& 2.82&2.90&2.95 & 2.97 \\
				\hline
				$\norm{{z}-{z}_h}_{0,\Omega}$& 3.8001e-03   &5.2896e-04   &6.9919e-05   &8.9948e-06   &1.1409e-06 \\
				\hline
				order&-& 2.84&2.92&2.96& 2.98 \\
				\hline
			\end{tabular}
		\end{center}
		\caption{Example \ref{example2}: Errors for the state $y$, adjoint state $z$, and the fluxes $\bm q$ and $\bm p$ when $k=1$.}\label{table_4}
	\end{table}
\end{example}

\section{Conclusions}

In our earlier work \cite{HuShenSinglerZhangZheng_HDG_Dirichlet_control5}, we considered an HDG method with degree $ k $ polynomials for all variables to approximate the solution of an optimal distributed control problems for an elliptic convection diffusion equation.  We proved optimal convergence rates for all variables in \cite{HuShenSinglerZhangZheng_HDG_Dirichlet_control5} when $ \bm \beta $ is divergence free; however, we did not obtain superconvergence.  In this work, we considered the same control problem and approximated the solution using a different HDG method with degree $ k + 1 $ polynomials for the flux variables and degree $ k $ polynomials for the other variables.  When $ k > 0 $ and $ \nabla \cdot \bm \beta \leq 0 $, we obtained superconvergence for the control, state, and dual state, and optimal convergence rates for the fluxes.  We plan to consider HDG methods for more complicated optimal control problems for PDEs in the future.


\section*{Acknowledgements}  W.\ Hu was supported in part by a postdoctoral fellowship for the annual program on Control Theory and its Applications at the Institute for Mathematics and its Applications (IMA) at the University of Minnesota.  J.\ Singler and Y.\ Zhang were supported in part by National Science Foundation grant DMS-1217122.  J.\ Singler and Y.\ Zhang thank the IMA for funding research visits, during which some of this work was completed.  X.\ Zheng thanks Missouri University of Science and Technology for hosting him as a visiting scholar; some of this work was completed during his research visit.  The authors thank Weifeng Qiu for many helpful discussions.


\bibliographystyle{spmpsci}
\bibliography{yangwen_ref_papers,yangwen_ref_books}
\end{document}